\documentclass{siamltex}
\usepackage{amsfonts,amsmath,amssymb}
\usepackage{mathrsfs,mathtools}
\usepackage{enumerate}
\usepackage{xspace}
\usepackage{hyperref}
\DeclareMathAlphabet{\mathpzc}{OT1}{pzc}{m}{it}

\newcommand{\R}{\mathbb{R}}
\newcommand{\Y}{\mathpzc{Y}}
\newcommand{\C}{\mathcal{C}}
\newcommand{\V}{\mathbb{V}}
\newcommand{\Wcal}{\mathcal{W}}
\newcommand{\Zcal}{\mathcal{Z}}
\newcommand{\Ucal}{\mathcal{U}}
\newcommand{\HL}{ \mbox{ \raisebox{7.2pt} {\tiny$\circ$} \kern-10.7pt} {H_L^1} }
\newcommand{\HLn}{{\mbox{\,\raisebox{4.7pt} {\tiny$\circ$} \kern-9.3pt}{H}^{1}_{L}  }}
\newcommand{\ue}{\mathscr{U}}
\DeclareMathOperator*{\tr}{tr_\Omega}
\newcommand{\Hsd}{\mathbb{H}^{-s}(\Omega)}
\newcommand{\calLs}{\mathcal{L}^s}
\newcommand{\GRAD}{\nabla}
\newcommand{\DIV}{\textrm{div}}
\newcommand{\diff}{\, \mbox{\rm d}}
\newcommand{\ie}{i.e.,\@\xspace}

\newcommand{\Hs}{\mathbb{H}^s(\Omega)}
\newcommand{\GL}{{\textup{\textsf{GL}}}}
\newcommand{\calG}{{\mathcal{G}}}

\newtheorem{remark}[theorem]{Remark}
\numberwithin{equation}{section}

\newcommand{\calL}{{\mathcal L}}
\DeclareMathOperator*{\esssup}{esssup}

\title{Regularity of solutions to space--time fractional wave equations: a PDE approach\thanks{EO is partially supported by CONICYT through FONDECYT project 3160201. AJS is partially supported by NSF grant DMS-1720213.}}

\author{Enrique Ot\'arola\thanks{Departamento de Matem\'atica, Universidad T\'ecnica Federico Santa Mar\'ia, Valpara\'iso, Chile. \texttt{enrique.otarola@usm.cl}.}
\and
Abner J.~Salgado\thanks{Department of Mathematics, University of Tennessee, Knoxville, TN 37996, USA.
\texttt{asalgad1@utk.edu}}}

\date{Draft version of \today.}

\begin{document}

\maketitle
\begin{abstract}
We consider an evolution equation involving the fractional powers, of order $s \in (0,1)$, of a symmetric and uniformly elliptic second order operator and Caputo fractional time derivative of order $\gamma \in (1,2]$. Since it has been shown useful for the design of numerical techniques for related problems, we also consider a quasi--stationary elliptic problem that comes from the realization of the spatial fractional diffusion as the Dirichlet-to-Neumann map for a nonuniformly elliptic problem posed on a semi--infinite cylinder.  We provide existence and uniqueness results together with energy estimates for both problems. In addition, we derive regularity estimates both in time and space; the time--regularity results show that the usual assumptions made in the numerical analysis literature are problematic.
\end{abstract}



\section{Introduction}
\label{sec:introduccion}

The purpose of this work is to derive regularity estimates for the solution to an initial boundary value problem for a space--time fractional wave equation. To make matters precise, let $\Omega$ be an open and bounded domain in $\R^n$ ($n\ge1$) with boundary $\partial\Omega$. Given $s \in (0,1)$, $\gamma \in (1,2]$, a forcing function $f$, and initial data $g$ and $h$, we seek $u$ such that
\begin{equation}
\label{eq:fractional_wave}
\begin{dcases}
  \partial^{\gamma}_t u + \mathcal{L}^s u  = f \ \text{ in } \Omega\times(0,T),
  \\
  u(0)  = g \ \text{ in } \Omega, \quad \partial_t u(0) = h \ \text{ in } \Omega.
\end{dcases}
\end{equation}
Here $\mathcal{L}$ denotes the linear, self--adjoint, second order, differential operator 
\begin{equation*}
\label{second_order}
 \mathcal{L} w = - \DIV_{x'} (A \nabla_{x'} w ) + c w,
\end{equation*}
supplemented with homogeneous Dirichlet boundary conditions; $0 \leq c \in L^\infty(\Omega)$ and $A \in C^{0,1}(\Omega,\GL(n,\R))$ is symmetric and uniformly positive definite. By $\calLs$, with $s \in (0,1)$, we denote the \emph{spectral} fractional powers of the operator $\calL$.

The fractional derivative in time $\partial^{\gamma}_t$ for $\gamma \in (1,2)$ is understood as \textit{the left--sided Caputo fractional derivative of order $\gamma$} with respect to $t$, which is defined by
\begin{equation}
\label{caputo}
\partial^{\gamma}_t u (t) := \frac{1}{\Gamma(2-\gamma)} \int_{0}^t \frac{1}{(t-r)^{\gamma - 1}} \frac{\diff^2 u(r)}{\diff r^2} \diff r,
\end{equation}
where $\Gamma$ is the Gamma function. For $\gamma = 2$, we consider the usual time derivative $\partial^2_{t}$.

The main source of difficulty in the analysis of \eqref{eq:fractional_wave} and in the design of efficient solution techniques is the nonlocality of both the fractional time derivative and the fractional space operator \cite{CS:07,fractional_book,Landkof,Samko,MR2270163}. When $\mathcal{L} = -\Delta$ and $\Omega = \R^n$, \ie in the case of the Laplacian in the whole space, Caffarelli and Silvestre \cite{CS:07} have overcome this difficulty by localizing $\mathcal{L}^s$ as the  Dirichlet-to-Neumann map for an extension problem to the upper half--space $\R_{+}^{n+1}$. The extension corresponds to a nonuniformly elliptic PDE. This important result was later extended in \cite{CT:10} and \cite{ST:10} to bounded domains $\Omega$ and more general operators, thereby obtaining an extension problem posed on the semi--infinite cylinder $\C:= \Omega \times (0,\infty)$; see also \cite{CDDS:11}. 
We shall use the Caffarelli--Silvestre extension result to rewrite our problem \eqref{eq:fractional_wave} as the following quasi--stationary elliptic problem with a dynamic boundary condition \cite{MR2600998,MR2737788,MR2954615,MR3192423}:
\begin{equation}
\label{eq:wave_alpha_extension}
\begin{dcases}
-\DIV \left( y^{\alpha} \mathbf{A} \nabla \ue \right) + y^{\alpha} c\ue = 0 \textrm{ in } \C \times(0,T), & \ue = 0 \textrm{ on }\partial_L \C  \times (0,T),\\
d_s \partial_{t}^{\gamma} \ue + \partial_{\nu}^{\alpha} \ue  = d_s f \textrm{ on } (\Omega \times \{ 0\}) \times (0,T),
\end{dcases}
\end{equation}
with the initial conditions
\begin{equation}
\label{eq:initial_cond}
 \ue = g \textrm{ on } \Omega \times \{ 0\}, ~t=0, \quad \partial_t \ue = h \textrm{ on } \Omega \times \{ 0\}, ~t=0,
\end{equation}
where $\partial_L \C= \partial \Omega \times [0,\infty)$ corresponds to the lateral boundary of $\C$, $\alpha =1-2s \in (-1,1)$, $d_s=2^\alpha \Gamma(1-s)/\Gamma(s)$ and the conormal exterior derivative of $\ue$ at $\Omega \times \{ 0 \}$ is
\begin{equation}
\label{eq:conormal_derivative}
\partial_{\nu}^{\alpha} \ue = -\lim_{y \rightarrow 0^+} y^\alpha \ue_y;
\end{equation}
the limit must be understood in the sense of distributions \cite{CS:07,CDDS:11,ST:10}. We will call $y$ the \emph{extended variable} and the dimension $n+1$ in $\R_+^{n+1}$ the \emph{extended dimension} of problem \eqref{eq:wave_alpha_extension}--\eqref{eq:initial_cond}. Finally, $\mathbf{A} =  \diag \{A,1\}  \in C^{0,1}(\C,\GL(n+1,\R))$. With the solution $\ue$ to the extension problem \eqref{eq:wave_alpha_extension}--\eqref{eq:initial_cond} at hand, we can find the solution to \eqref{eq:fractional_wave} via \cite{CT:10,CS:07,CDDS:11,ST:10}:
\[
 u = \ue|_{y=0}.
\]

To the best of the authors' knowledge this is the first work that analyzes problem \eqref{eq:fractional_wave} and its extended version \eqref{eq:wave_alpha_extension}: we provide existence and uniqueness results and derive regularity estimates both in time and space. Concerning spatial regularity of the extended problem \eqref{eq:wave_alpha_extension}, we derive such estimates in weighted Sobolev spaces. In particular, we establish analytic regularity with respect to the extended variable $y \in (0,\infty)$. We prove that $\ue$ belongs to countably normed, power--exponentially weighted Bochner spaces of analytic functions with respect to $y$, taking values in Sobolev spaces in $\Omega$. Our main motivation for deriving such regularity results is the fact that any rigorous study of a numerical scheme to approximate the solution to \eqref{eq:fractional_wave} via the resolution of problem \eqref{eq:wave_alpha_extension} must be concerned with the regularity of its solution. In fact, as it is well known, smoothness and rate of approximation go hand in hand. This is exactly the content of direct and converse theorems in approximation theory \cite{MR1217081,MR0145254}. An instance of this is provided in \cite{BMNOSS:17,MFSV:17} where, in the case of an elliptic equation involving the spectral fractional Laplacian, the analytic regularity of the solution to the extended problem with respect to the extended direction has been shown essential to provide efficient solution techniques.

In the literature, several numerical techniques have been proposed to approximate the solution to problems involving the Caputo fractional derivative of order $\gamma \in (1,2)$. To the best of our knowledge the first work that proposes a scheme based on finite differences is \cite{MR2200938}. In this work, the authors prove that the consistency error of such scheme is $\mathcal{O}(\tau^{3-\gamma})$, where $\tau$ denotes the time step. However, this error estimate requires a rather strong regularity assumption, namely that \emph{the third time derivative of $u$ is continuous on $[0,T]$}. This assumption is problematic and it has been largely ignored in the literature \cite{MR3549087,MR2651618,MR2870026,MR2970754}. Since $\gamma \in (1,2)$, the second and higher order derivatives of the solution $u$ of \eqref{eq:fractional_wave} with respect to $t$ are unbounded as $t \downarrow 0$. In this work, we examine the singular behavior of $\partial_t^2 u$ and $\partial_{t}^3 u$ when $t \downarrow 0$ and derive realistic time--regularity estimates for $u$ and $\ue$ in Theorems \ref{thm:time_regularity} and \ref{thm:time_regularity_extension}, respectively. These refined regularity estimates are  of importance in the analysis of time discretization schemes for problems \eqref{eq:fractional_wave} and \eqref{eq:wave_alpha_extension}.

The outline of this paper is as follows. In section \ref{sec:Prelim} we introduce some terminology used throughout this work. We recall the spectral definition of fractional powers of elliptic operators in section \ref{sub:fractional_L} and, in section \ref{sub:CaffarelliSilvestre}, state the essential result by Caffarelli and Silvestre about their localization. In sections \ref{sub:fractional_derivatives} and \ref{sub:ML}, we  introduce some elements of fractional calculus that will be important in the analysis that follows. Upon introducing suitable definitions of weak solutions, we derive the well--posedness of problems \eqref{eq:fractional_wave} and \eqref{eq:wave_alpha_extension}--\eqref{eq:initial_cond} in sections \ref{sub:existunique} and \ref{sub:extended}, respectively. In section \ref{sub:time_regularity} we derive time regularity results for problems \eqref{eq:fractional_wave} and \eqref{eq:wave_alpha_extension}--\eqref{eq:initial_cond}, while, in section \ref{sub:space_regularity} we study spatial regularity properties for the solution to the extended problem \eqref{eq:wave_alpha_extension}--\eqref{eq:initial_cond}. Space-time regularity of the solution to \eqref{eq:wave_alpha_extension}--\eqref{eq:initial_cond} is discussed in section \ref{sub:spacetime}. We conclude the discussion with section \ref{sub:application} where we sketch how the obtained regularity results can be applied to the derivation of error estimates for fully discrete schemes for \eqref{eq:wave_alpha_extension}--\eqref{eq:initial_cond}.

\section{Notation and preliminaries}
\label{sec:Prelim}

Throughout this work $\Omega$ is an open, bounded, and convex subset of $\R^n$ ($n\geq1$) with boundary $\partial\Omega$. We will follow the notation of \cite{NOS,Otarola} and define the semi--infinite cylinder $\C := \ \Omega \times (0,\infty)$ and its lateral boundary $\partial_L \C  := \partial \Omega \times [0,\infty)$. For $\Y>0$, we define the truncated cylinder with base $\Omega$ and height $\Y$ as $\C_\Y := \Omega \times (0, \Y)$; its lateral boundary is denoted by $\partial_L \C_{\Y}  = \partial \Omega \times (0,\Y)$. Since we will be dealing with objects defined in $\R^n$ and $\R^{n+1}$, it will be convenient to distinguish the extended $(n+1)$--dimension: if $x\in \R^{n+1}$, we write
$
x =  (x',y),
$
with $x' \in \R^n$ and $y\in\R$. 

Whenever $X$ is a normed space, $X'$ denotes its dual and $\|\cdot\|_{X}$ its norm. If, in addition, $Y$ is a normed space, we write $X \hookrightarrow Y$ to indicate continuous embedding. We will follow standard notation for function spaces \cite{Adams,Tartar}. For an open set $D \subset \R^d$ ($d \geq 1$), if $\omega$ is a weight and $p \in (1,\infty)$ we denote the Lebesgue space of $p$-integrable functions with respect to the measure $\omega \diff x$ by $L^p(\omega,D)$ \cite{HKM,Kufner80,Turesson}. Similar notation will be used for weighted Sobolev spaces. If $T >0$ and $\phi: D \times(0,T) \to \R$, we consider $\phi$ as a function of $t$ with values in a Banach space $X$,
$
 \phi:(0,T) \ni t \mapsto  \phi(t) \equiv \phi(\cdot,t) \in X
$.
For $1 \leq p \leq \infty$, $L^p( 0,T; X)$ is the space of $X$-valued functions whose norm in $X$ is in $L^p(0,T)$. This is a Banach space for the norm
\[
  \| \phi \|_{L^p( 0,T;X)} = \left( \int_0^T \| \phi(t) \|^p_X \diff t \right)^{\hspace{-0.1cm}\frac{1}{p}} 
  , \quad 1 \leq p < \infty, \quad
  \| \phi \|_{L^\infty( 0,T;X)} = \esssup_{t \in (0,T)} \| \phi(t) \|_X.
\]

The relation $a \lesssim b$ means $a \leq Cb$, with a constant $C$ that neither depends on $a$ or $b$. The value of $C$ might change at each occurrence.

Finally, since we assume $\Omega$ to be convex, in what follows we will make use, without explicit mention, of the following regularity result \cite{Grisvard}:
\begin{equation}
\label{eq:Omega_regular}
\| w \|_{H^2(\Omega)} \lesssim \| \mathcal{L} w \|_{L^2(\Omega)} \quad \forall w \in H^2(\Omega) \cap H^1_0(\Omega).
\end{equation}


\subsection{Fractional powers of second order elliptic operators}
\label{sub:fractional_L}

We adopt the \emph{spectral} definition for the fractional powers of the operator $\mathcal{L}$. This is, to define $\mathcal{L}^s$, we invoke spectral theory for the operator $\calL$ \cite{BS}. The eigenvalue problem: Find $(\lambda,\varphi) \in \R \times H_0^1(\Omega) \setminus \{ 0\}$ such that
\begin{equation}
\label{eq:eigenpairs}
    \mathcal{L} \varphi = \lambda \varphi  \text{ in } \Omega,
    \qquad
    \varphi = 0 \text{ on } \partial\Omega,
\end{equation}
has a countable collection of solutions $\{ \lambda_{\ell}, \varphi_{\ell} \}_{\ell\in \mathbb N} \subset \R_+ \times H_0^1(\Omega)$ such that $\{\varphi_{\ell} \}_{\ell=1}^{\infty}$ is an orthonormal basis of $L^2(\Omega)$ and an orthogonal basis of $H_0^1(\Omega)$, for the inner product induced by $\calL$. With these eigenpairs at hand, we introduce, for $s \geq 0$, the space
\begin{equation}
\label{def:Hs}
  \Hs = \left\{ w = \sum_{\ell=1}^\infty w_\ell \varphi_\ell: \| w \|^2_{\Hs}:= \sum_{\ell=1}^\infty \lambda_\ell^s |w_\ell|^2 < \infty \right\},
\end{equation}
and denote by $\Hsd$ the dual space of $\Hs$. The duality pairing between the aforementioned spaces will be denoted by $\langle \cdot, \cdot \rangle$. 

We notice that, if  $s \in (0,\tfrac12)$, $\Hs = H^s(\Omega) = H_0^s(\Omega)$, while, for $s \in (\tfrac12,1)$, $\Hs$ can be characterized by \cite{Lions,McLean,Tartar}
\[
  \Hs = \left\{ w \in H^s(\Omega): w = 0 \text{ on } \partial\Omega \right\}.
\]
If $s = \frac{1}{2}$, we have that $\mathbb{H}^{\frac{1}{2}}(\Omega)$ is the so--called Lions--Magenes space $H_{00}^{\frac{1}{2}}(\Omega)$ \cite{Lions,Tartar}. If $s\in(1,2]$, owing to \eqref{eq:Omega_regular}, we have that $\Hs = H^s(\Omega)\cap H^1_0(\Omega)$ \cite{ShinChan}.

The fractional powers of the operator $\mathcal L$ are thus defined by
\begin{equation}
  \label{def:second_frac}
 \mathcal{L}^s: \Hs \to \Hsd, \quad  \mathcal{L}^s w  := \sum_{\ell=1}^\infty \lambda_\ell^{s} w_\ell \varphi_\ell,  \quad s \in (0,1).
\end{equation} 

\subsection{Weighted Sobolev spaces}
\label{sub:CaffarelliSilvestre}
Both extensions, the one by Caffarelli and Silvestre \cite{CS:07} and the ones in \cite{CT:10,CDDS:11,ST:10} for $\Omega$ bounded and general elliptic operators require us to deal with a local but nonuniformly elliptic problem. To provide an analysis it is thus suitable to define the weighted Sobolev space
\begin{equation}
  \label{HL10}
  \HL(y^{\alpha},\C) = \left\{ w \in H^1(y^\alpha,\C): w = 0 \textrm{ on } \partial_L \C\right\}.
\end{equation}
Since $\alpha \in (-1,1)$, $|y|^\alpha$ belongs to the Muckenhoupt class $A_2$ \cite{MR1800316,Muckenhoupt,Turesson}. We thus have the following important consequences: $H^1(y^{\alpha},\C)$ is a Hilbert space and $C^{\infty}(\Omega) \cap H^1(y^{\alpha},\C)$ is dense in $H^1(|y|^{\alpha},\C)$ (cf.~\cite[Proposition 2.1.2, Corollary 2.1.6]{Turesson}, \cite{KO84} and \cite[Theorem~1]{GU}).

As \cite[inequality (2.21)]{NOS} shows, the following \emph{weighted Poincar\'e inequality} holds:
\begin{equation}
\label{Poincare_ineq}
\| w \|_{L^2(y^{\alpha},\C)} \lesssim \| \nabla w \|_{L^2(y^{\alpha},\C)} \quad \forall w \in \HL(y^{\alpha},\C).
\end{equation}
Thus, $\| \nabla w \|_{L^2(y^{\alpha},\C)}$ is equivalent to the norm in $\HL(y^{\alpha},\C)$. For $w \in H^1( y^{\alpha},\C)$, $\tr w$ denotes its trace onto $\Omega \times \{ 0 \}$. We recall that, for $\alpha = 1-2s$, \cite[Prop. 2.5]{NOS} yields
\begin{equation}
\label{Trace_estimate}
\tr \HL(y^\alpha,\C) = \Hs,
\qquad
  \|\tr w\|_{\Hs} \lesssim \| w \|_{\HLn(y^\alpha,\C)}.
\end{equation}

The seminal work of Caffarelli and Silvestre \cite{CS:07} and its extensions to bounded domains \cite{CT:10,CDDS:11,ST:10} showed that the operator $\mathcal{L}^s$ can be realized as the Dirichlet-to-Neumann map for a nonuniformly elliptic boundary value problem. Namely, if $\Ucal \in \HL(y^{\alpha},\C)$ solves
\begin{equation}
\label{eq:extension}
    -\DIV\left( y^\alpha \mathbf{A} \GRAD \Ucal \right) + c y^\alpha  \Ucal= 0  \text{ in } \C, \quad
    \Ucal= 0  \text{ on } \partial_L \C, \quad
    \partial_\nu^\alpha \Ucal = d_s f  \text{ on } \Omega \times \{0\},
\end{equation}
where $\alpha = 1-2s$, $\partial_\nu^\alpha \Ucal = -\lim_{y\downarrow 0} y^\alpha \Ucal_y$ and $d_s = 2^\alpha \Gamma(1-s)/\Gamma(s)$ is a normalization constant, then $\mathfrak{u} = \tr \Ucal \in \Hs$ solves
\begin{equation}
\label{eq:fraclap}
\mathcal{L}^s \mathfrak{u} = f.
\end{equation}

\subsection{Fractional derivatives and integrals}
\label{sub:fractional_derivatives}
The left--sided Caputo fractional derivative of order $\gamma \in (1,2)$ is defined as in \eqref{caputo}.

Given a function $g \in L^1(0,T)$, the left Riemann--Liouville fractional integral $I^{\sigma}[g]$ of order $\sigma>0$ is defined by \cite[formula (2.1.1)]{fractional_book}, \cite[formula (2.17)]{Samko}
\begin{equation}
\label{fractional_integral}
I^{\sigma}[g](t) = \frac{1}{\Gamma(\sigma)} \int_{0}^t \frac{g(r)}{(t-r)^{1-\sigma}} \diff r.
\end{equation}

Young's inequality for convolutions immediately yields the following result.
\begin{lemma}[continuity]
\label{le:continuity}
If $g \in L^2(0,T)$ and $\phi \in L^1(0,T)$, then the operator
\[
 g \mapsto \Phi, \qquad \Phi(t) = \phi \star g (t) = \int_0^t \phi(t-r) g(r) \diff r
\]
is continuous from $L^2(0,T)$ into itself and 
$
 \| \Phi \|_{L^2(0,T)} \leq \| \phi \|_{L^1(0,T)} \| g \|_{L^2(0,T)}.
$
\end{lemma}
\begin{corollary}[continuity of $I^{\sigma}$]
\label{co:continuity}
For any $\sigma > 0$, the left Riemann-Liouville fractional integral $I^{\sigma} g$ is continuous from
$L^2(0,T)$ into itself and
\[\|I^{\sigma} [g]\|_{L^2(0,T)}\le \frac{T^\sigma}{\Gamma(\sigma+1)}
\|g\|_{L^2(0,T)} \quad\forall g\in L^2(0,T).
\]
\end{corollary}


\subsection{The Mittag-Leffler function} 
\label{sub:ML}

For $\gamma > 0$ and $\mu \in \mathbb{C}$, we define the two--parameter Mittag--Leffler function $E_{\gamma,\mu}$ as \cite{MR3244285,fractional_book,Podlubny} 
\begin{equation}
\label{def:ML}
 E_{\gamma,\mu}(z):= \sum_{k=0}^{\infty} \frac{z^k}{\Gamma( \gamma k + \mu )}, \quad z \in \mathbb{C},
\end{equation}
which generalizes the classical Mittag--Leffler function $E_{\gamma}(z):= E_{\gamma,1}(z)$. It can be shown that $E_{\gamma,\mu}(z)$ is an entire function of $z \in \mathbb{C}$. In what follows the following members of the family $\{ E_{\gamma,\mu}: \gamma > 0, \mu \in \R\}$ will be of importance: $E_{\gamma,1}$, $E_{\gamma,2}$, and $E_{\gamma,\gamma}$; they will allow us to derive solution representation formulas for \eqref{eq:fractional_wave} and \eqref{eq:wave_alpha_extension}--\eqref{eq:initial_cond} in the case that $\gamma \in (1,2)$. 

For $\lambda,\gamma,t \in \R^+$, we have \cite[Lemma 2.23]{fractional_book}
\begin{equation}
\label{dEbeta1}
\partial_t^{\gamma} E_{\gamma,1}(-\lambda t^{\gamma}) = -\lambda E_{\gamma,1}(-\lambda t^{\gamma}).
\end{equation}
If $\gamma \in (0,2)$, $\mu \in \R$, $\pi \gamma/2 < \delta < \min\{ \pi, \pi \gamma \}$ and $\delta \leq |\arg(z) | \leq \pi$, then 
\begin{equation}
\label{ML_estimate}
|E_{\gamma,\mu}(z)| \lesssim  (1 + |z|)^{-1},
\end{equation}
where the hidden constant depends only on $\gamma$, $\mu$ and $\delta$; see \cite[Section 1.8]{fractional_book} and \cite[Theorem 1.6]{Podlubny}.

For $q \in \mathbb{N}$ and $\gamma, \mu - q \in \R^{+}$ the following differentiation formula follows from \cite[formula (1.83)]{Podlubny}
\begin{equation}
\label{eq:derivative_t_ML}
\left(\frac{\diff}{\diff t}\right)^q \left( t^{\mu-1}E_{\gamma,\mu}(t^{\gamma}) \right) = t^{\mu-q-1} E_{\gamma,\mu - q}(t^{\gamma}).
\end{equation}
As an application of the previous formula we record, for future reference, the following particular cases: If $t,\gamma,\lambda \in \mathbb{R}^+$, then
\begin{equation}
\label{eq:G_derivatives2}
\frac{\diff}{\diff t} \left( t E_{\gamma,2}(-\lambda t^{\gamma}) \right)=  E_{\gamma,1}(-\lambda t^{\gamma}),
\end{equation}
and
\begin{equation}
\label{eq:G_derivatives3}
  \frac{\diff}{\diff t} \left( t^{\gamma-1} E_{\gamma,\gamma}(-\lambda t^{\gamma}) \right)=  t^{\gamma-2} E_{\gamma,\gamma-1}(-\lambda t^{\gamma});
\end{equation}
see also \cite[Lemma 2.2]{MR3613323} and \cite[Lemma 3.2]{Sakayama}.

We conclude this section with the following formula that follows from \cite[Lemma 3.2]{Sakayama} and \cite[Lemma 2.2]{MR3613323}: If $t,\gamma, \lambda \in \R^{+}$ and $q$ denotes a positive integer, then
\begin{equation}
\label{eq:G_derivatives}
\left(\frac{\diff}{\diff t}\right)^q  E_{\gamma,1}(-\lambda t^{\gamma}) = - \lambda t^{\gamma - q} E_{\gamma,\gamma-q+1}(-\lambda t^{\gamma}).
\end{equation}

\section{Well--posedness and energy estimates} 
\label{sec:wellposedness}
In this section we will study the existence and uniqueness of weak solutions to problems \eqref{eq:fractional_wave} and \eqref{eq:wave_alpha_extension}--\eqref{eq:initial_cond} and derive energy estimates. We will begin with the analysis of the fractional wave equation \eqref{eq:fractional_wave} and distinguish two cases: $\gamma = 2$ and $\gamma \in (1,2)$; the first one being considerable simpler.

\subsection{The fractional wave equation}
\label{sub:existunique}

\subsubsection{Case $\gamma=2$} We assume that the data of problem \eqref{eq:fractional_wave} is such that $f \in L^2(0,T;L^2(\Omega))$, $g \in \Hs$ and $h \in L^2(\Omega)$.

\begin{definition}[weak solution for $\gamma = 2$]
\label{def:weak_2}
We call $u \in L^2(0,T;\Hs)$, with $\partial_t u \in L^2(0,T;L^2(\Omega))$ and $\partial_{t}^2 u \in L^2(0,T;\Hsd)$, a weak solution of problem \eqref{eq:fractional_wave} if $u(0) = g$, $\partial_t u(0) = h$ and a.e. $t \in (0,T)$,
\[
  \langle \partial_t^2 u , v \rangle + \langle \mathcal{L}^s u , v \rangle
  = \langle f , v \rangle \quad \forall v \in \Hs,
\]
where $\langle \cdot , \cdot \rangle$ denotes the the duality pairing between $\Hs$ and $\Hsd$.

\end{definition}
\begin{remark}[initial conditions $\gamma=2$]\rm
Since a weak solution $u$ of \eqref{eq:fractional_wave} satisfies that $u \in L^2(0,T;\Hs)$, $\partial_t u \in L^2(0,T;L^2(\Omega))$ and $\partial_{t}^2 u \in L^2(0,T;\Hsd)$, an application of \cite[Lemma 7.3]{MR3014456} implies that $u \in C([0,T]; L^2(\Omega))$ and that $\partial_t u \in C([0,T]; \Hsd)$. Consequently, the initial conditions involved in Definition \ref{def:weak_2} are appropriately defined.
\label{rem:initial_cond_2}
\end{remark}

Existence and uniqueness of a weak solution, in the sense of Definition \ref{def:weak_2}, to problem \eqref{eq:fractional_wave}, together with energy estimates are obtained by slightly modifying the arguments, based on a Galerkin technique, of \cite{Evans,Lions,MR3014456}. For brevity, we present the following result and leave the details to the reader. We only mention, since this will be used below, that the solution has the representation $u(x',t) = \sum_{k\geq 1} u_k(t) \varphi_k(x')$, where the coefficients $u_k$ are given by
\begin{equation}
\label{eq:var_parameters}
  u_k(t) = g_k \cos(\lambda_k^{s/2} t) + \frac{h_k}{\lambda_k^{s/2}} \sin(\lambda_k^{s/2}t) 
  + \frac1{\lambda_k^{s/2}} \int_0^t f_k(r) \sin(\lambda_k^{s/2}(t-r)) \diff r.
\end{equation}

\begin{theorem}[well--posedness for $\gamma = 2$] 
\label{thm:wp_2}
Given $s \in (0,1)$, $\gamma = 2$, $f \in L^2(0,T;L^2(\Omega))$, $g \in \Hs$ and $h \in L^2(\Omega)$, problem \eqref{eq:fractional_wave} has a unique weak solution. In addition, 
\end{theorem}
\begin{equation}\label{eq:energy_estimate_2}
\| u\|_{L^{\infty}(0,T;\Hs)} + \| \partial_t u\|_{L^{\infty}(0,T;L^2(\Omega))} \lesssim \| f \|_{L^2(0,T;L^2(\Omega))} + \| g \|_{\Hs} + \| h \|_{L^2(\Omega)},
\end{equation}
where the hidden constant is independent of the problem data.

\subsubsection{Case $\gamma \in (1,2)$} 
\label{sec:fractional_wave_equation}
The analysis of fractional evolution differential equations as \eqref{eq:fractional_wave} and the derivation of energy estimates is not an easy task. The main technical difficulty lies in the fact that an exact integration by parts formula is not available. This has two main consequences. First, it is difficult to introduce a definition of weak solutions to \eqref{eq:fractional_wave} in the sense of distributions \cite{MR3613323}. Second, when deriving an energy estimates boundary terms need to be estimated. 

In this section, we follow \cite{MR3613323,Sakayama} and obtain the existence and uniqueness of a suitable weak solution to problem \eqref{eq:fractional_wave} together with energy estimates. The analysis is based on a solution representation formula. To describe it, we invoke the eigenpairs 
defined in Section \ref{sub:fractional_L}, and formally write the solution to problem \eqref{eq:fractional_wave} as follows:
\begin{equation}
u(x',t) = \sum_{k \geq 1} u_k(t) \varphi_k(x').
\label{eq:formal_series}
\end{equation}
Since, at this formal stage, we have $u(x',0) = g(x')$ and $\partial_t u(x',0) = h(x')$, this representation yields the following fractional initial value problem for $u_k$:
\begin{equation} 
 \label{eq:uk}
    \partial^{\gamma}_t u_k(t) + \lambda_k^s u_k(t) = f_k(t), \, t > 0, \quad
     u_k(0) = g_k, \quad \partial_t u_k(0) = h_k, \quad k \in \mathbb{N},
\end{equation}
where $g_k = (g,\varphi_k)_{L^2(\Omega)}$, $h_k = (h,\varphi_k)_{L^2(\Omega)}$
and $f_{k}(t) = (f(\cdot,t),\varphi_k)_{L^2(\Omega)}$. The theory of fractional ordinary differential equations \cite{fractional_book,Samko} gives a unique function $u_k$ satisfying problem \eqref{eq:uk}. In addition, an explicit representation formula for the solution $u_k$ holds (see, for instance, \cite[formula (2.1)]{MR3613323} or \cite{Sakayama}):
\begin{multline}
\label{eq:u_k_rs}
 u_k(t) =  E_{\gamma,1} (-\lambda_{k}^s t^{\gamma}) g_k + t E_{\gamma,2} (-\lambda_{k}^s t^{\gamma}) h_k 
 \\
 + \int_{0}^t (t-r)^{\gamma-1}E_{\gamma,\gamma}(-\lambda_k^s(t-r)^{\gamma})f_k(r)\diff r.
\end{multline}

With the help of identities \eqref{eq:G_derivatives2}, \eqref{eq:G_derivatives3}, and \eqref{eq:G_derivatives} we can derive formulas for the derivatives, of order $1$ and $\gamma \in (1,2)$, of $u_k$  which shall prove useful in the analysis that follows:
\begin{multline}
 \frac{\diff u_k}{\diff t} (t) = -\lambda_k^s t^{\gamma-1} E_{\gamma,\gamma}(-\lambda_k^s t^{\gamma})g_k 
 \\ 
 + E_{\gamma,1}(-\lambda_k^s t^{\gamma})h_k + \int_{0}^t (t-r)^{\gamma-2} E_{\gamma,\gamma-1}(-\lambda_k^s(t-r)^{\gamma} ) f_k(r) \diff r
 \label{eq:du_k_rs}
\end{multline}
and
\begin{multline}
 \partial_t^{\gamma} u_k(t) = -\lambda_k^s E_{\gamma,1}(-\lambda_k^s t^{\gamma})g_k  - t \lambda_k^s E_{\gamma,2}(-\lambda_k^s t^{\gamma})h_k
 \\ 
 -\lambda_k^s \int_{0}^t (t-r)^{\gamma-1} E_{\gamma,\gamma}(-\lambda_k^s(t-r)^{\gamma} ) f_k(r) \diff r + f_k(t);
 \label{eq:dgammau_k_rs}
\end{multline}
see \cite[Lemma 2.3]{MR3613323} and \cite[formulas (3.8) and (3.12)]{Sakayama}, respectively.

We now proceed to make rigorous the previous formal considerations and show that the expression \eqref{eq:formal_series}, where $u_k(t)$ are given by \eqref{eq:u_k_rs}, converges. This gives rise to a notion of weak solution for \eqref{eq:fractional_wave} which will satisfy, on the basis of \eqref{eq:u_k_rs}, the solution representation formula
\begin{equation}
 \label{eq:solution_to_fractional_wave}
 u(x',t) = G_{\gamma}(t) g + H_{\gamma}(t) h + \int_{0}^t F_{\gamma}(t-r)f(x',r)\diff r,
\end{equation}
where the solution operators $G_{\gamma}$, $H_{\gamma}$ and $F_{\gamma}$ are defined as follows:
For $f \equiv 0$ and $h \equiv 0$, and $f \equiv 0$ and $g \equiv 0$, we define, respectively,
\begin{equation}
 \label{eq:sol_rep_GH}
 G_{\gamma}(t) w = \sum_{k = 1}^{\infty} E_{\gamma,1} (-\lambda_{k}^s t^{\gamma}) w_k  \varphi_k(x'),
\quad
 H_{\gamma}(t) w = \sum_{k = 1}^{\infty} t E_{\gamma,2} (-\lambda_{k}^s t^{\gamma}) w_k  \varphi_k(x'),
\end{equation}
where $w_k = (w, \varphi_k)_{L^2(\Omega)}$. We also define the solution operator for $g \equiv 0$ and $h \equiv 0$ as
\begin{equation}
 \label{eq:sol_rep_F}
 F_{\gamma}(t) w = \sum_{k = 1}^{\infty} t^{\gamma-1} E_{\gamma,\gamma} (-\lambda_{k}^s t^{\gamma}) w_k  \varphi_k(x').
\end{equation}
We immediately comment that, in view of the estimate \eqref{ML_estimate}, for all $t > 0$, $G_{\gamma}(t)$, $F_{\gamma}(t)$ and $H_{\gamma}(t)$ belong to $\mathcal{B}(L^2(\Omega))$ \cite{MR2595950,Podlubny,Sakayama}.

We now introduce the following notion of weak solution.

\begin{definition}[weak solution for $\gamma \in (1,2)$]
\label{def:weak_waveg12}
We call $u \in L^{\infty}(0,T;\Hs)$, with $\partial_t u \in L^{\infty}(0,T;L^2(\Omega))$ and $\partial_t^{\gamma} u \in L^2(0,T;\Hsd)$, a weak solution of problem \eqref{eq:fractional_wave} if $u(0) = g$, $\partial_t u (0) = h$, and for a.e. $t \in (0,T)$,
 \[
  \langle \partial_t^{\gamma} u , v \rangle + \langle \mathcal{L}^s u , v \rangle
  = \langle f , v \rangle \quad \forall v \in \Hs,
\]
where $\langle \cdot , \cdot \rangle$ denotes the the duality pairing between $\Hs$ and $\Hsd$.
\end{definition}

Notice that in the previous definition we are not being specific about the smoothness in the problem data, nor the sense in which the initial conditions are understood. The following result addresses these issues and provides existence and uniqueness of a weak solution.

\begin{theorem}[well--posedness for $\gamma \in (1,2)$] 
\label{thm:wp_gamma}
Let $s \in (0,1)$, $\gamma \in (1,2)$, $f \in L^{\infty}(0,T;L^2(\Omega))$, $g \in \mathbb{H}^{2s}(\Omega)$ and $h \in L^2(\Omega)$. Then, problem \eqref{eq:fractional_wave} has a unique weak solution which is given by the solution representation formula \eqref{eq:solution_to_fractional_wave}. In addition,
\end{theorem}
\begin{equation}\label{eq:energy_estimate_gamma_1}
\| u\|_{L^{\infty}(0,T;\Hs)} + \| \partial_t u\|_{L^{2}(0,T;L^2(\Omega))} 
\lesssim \| f \|_{L^2(0,T;L^2(\Omega))} + \| g \|_{\Hs} + \| h \|_{L^2(\Omega)},
\end{equation}
and
\begin{equation}\label{eq:energy_estimate_gamma_2}
\| \partial_t u\|_{L^{\infty}(0,T;L^2(\Omega))} \lesssim \| f \|_{L^{\infty}(0,T;L^2(\Omega))} + \| g \|_{\mathbb{H}^{2s}(\Omega)} + \| h \|_{L^2(\Omega)}.
\end{equation}
The hidden constants, in both inequalities, are independent of the problem data.
\begin{proof}
Let us first show that the function given by the representation formula \eqref{eq:solution_to_fractional_wave} has the requisite smoothness and is indeed a solution to \eqref{eq:fractional_wave}. This will be done, with the aid of representation formulas \eqref{eq:u_k_rs} and \eqref{eq:du_k_rs}, in four steps.

\noindent \boxed{\emph{Step 1.}} The goal of this step is to analyze the convergence of $\sum_{k\geq1} u_k\varphi_k$. To accomplish this task, we let $m,n \in \mathbb{N}$ with $n \geq m$. We thus invoke the representation formula \eqref{eq:u_k_rs} for $u_k(t)$ and the estimate \eqref{ML_estimate} to conclude, for $t \in [0,T]$, that
\begin{equation}
  \begin{aligned}
 \left\| \sum_{k=m}^n u_k(t) \varphi_k \right\|^2_{\Hs} &= \sum_{k=m}^{n} \lambda_k^s u_k(t)^2 \lesssim 
  \sum_{k=m}^n \frac{\lambda_k^s}{(1+\lambda_k^s t^\gamma)^2} g_k^2  \\ &+  t^{2-\gamma} \sum_{k=m}^n   \frac{\lambda_k^s t^{\gamma}}{(1 + \lambda_k^s t^{\gamma})^2} h_k^2 
  \\
  &+ \sum_{k=m}^{n} \lambda_k^s \left( \int_0^{t} (t-r)^{\gamma-1} E_{\gamma,\gamma} (-\lambda_{k}^s (t-r)^{\gamma})f_k(r) \diff r\right)^2.
  \end{aligned}
  \label{eq:first_step}
\end{equation} 
We denote the last term on the right--hand side of the previous expression by $\textrm{I}$ and bound it by using the Cauchy--Schwarz inequality and the estimate \eqref{ML_estimate}. We thus obtain that
\begin{multline}
  \textrm{I}  \lesssim \sum_{k=m}^{n}  \left( \int_0^t \frac{\lambda_k^s(t-r)^{\gamma}}{ ( 1+ \lambda_k^s(t-r)^{\gamma})^2}  (t-r)^{\gamma-2} \diff r\right)  \left( \int_0^t f_k^2(r) \diff r \right) 
  \\
  \lesssim \sum_{k=m}^{n} \left( \int_0^t (t-r)^{\gamma-2} \diff r\right)\left( \int_0^T f_k^2(t) \diff t \right)
  \lesssim \sum_{k=m}^{n} t^{\gamma-1} \int_0^T f_k^2(t) \diff t .
\end{multline}
Replacing this estimate into \eqref{eq:first_step} we arrive at
\begin{equation}
 \sup_{t \in [0,T]} \sum_{k=m}^{n} \lambda_k^s u_k(t)^2 \lesssim  \sum_{k=m}^n \lambda_k^s g_k^2 + T^{2-\gamma} \sum_{k=m}^n h_k^2 + T^{\gamma -1}\sum_{k=m}^{n} \int_0^T f_k^2(t) \diff t.
 \label{eq:energy_1}
\end{equation}
Consequently, since $\gamma \in (1,2)$, $g \in \Hs$, $h \in L^2(\Omega)$ and $f \in L^2(0,T;L^2(\Omega))$, the previous estimate allows us to obtain that
\begin{equation}
 \lim_{m,n \rightarrow \infty} \sup_{t \in [0,T] }\left\| \sum_{k=m}^n u_k(t) \varphi_k \right\|_{\Hs}  = 0,
 \label{eq:limit_1}
\end{equation}
\ie $u \in L^\infty(0,T;\Hs)$.

\noindent \boxed{\emph{Step 2.}} In this step we study the convergence of the formal expression $\sum_{k \geq 1} \partial_t u_k\varphi_k$. Let $m,n \in \mathbb{N}$ with $n \geq m$. 
In view of the representation formula \eqref{eq:du_k_rs} for $\partial_t u_k(t)$, similar arguments to the ones used in Step 1 allow us to obtain that
\begin{multline}
 \left\| \sum_{k=m}^n \partial_t u_k(t) \varphi_k \right\|_{L^2(0,T;L^2(\Omega))}^{2} 
 \lesssim \int_{0}^T t^{\gamma-2}\left( \sum_{k=m}^n  \frac{ \lambda_k^{s}t^{\gamma} }{(1 + \lambda_k^s t^{\gamma})^2} \lambda_k^s g_k^2 \right) \diff t 
 \\
 + T \sum_{k=m}^n  h_k^2 + \int_0^T  \sum_{k=m}^{n} \left( \int_0^{t} (t-r)^{\gamma-2} E_{\gamma,\gamma-1} (-\lambda_{k}^s (t-r)^{\gamma})f_k(r) \diff r\right)^2 \diff t.
 \label{eq:second_step}
\end{multline}
Let us denote the last term on the right--hand side of the previous expression by $\textrm{II}=\sum_{k=m}^n \textrm{II}_k$. With this notation it is immediate to see that $\textrm{II}_k$ is the square of the $L^2(0,T)$--norm of the convolution of the functions $\eta_k(t) = t^{\gamma-2}E_{\gamma,\gamma-1}(-\lambda_k^s t^\gamma)$ and $f_k$. Therefore, by Young's inequality for convolutions, we obtain that
\begin{align*}
  \textrm{II}_k \leq \| \eta_k \|_{L^1(0,T)}^2 \| f_k \|_{L^2(0,T)}^2.
\end{align*}
In addition, \eqref{ML_estimate} allows us to see that
\[
  \| \eta_k \|_{L^1(0,T)} = \int_0^T t^{\gamma-2}|E_{\gamma,\gamma-1}(-\lambda_k^s t^\gamma)| \diff t \lesssim \int_0^T t^{\gamma -2} \diff t \lesssim T^{\gamma-1}.
\]
In conclusion, we have proved the bound
\[
\textrm{II} \lesssim T^{2(\gamma-1)} \sum_{k=m}^{n} \int_0^T f_k^2(t) \diff t .
\]
We thus replace the previous estimate into \eqref{eq:second_step} and conclude that
\begin{equation}
 \int_{0}^T \sum_{k=m}^n |\partial_t u_k(t)|^2\diff t  \lesssim T^{\gamma-1}  \sum_{k=m}^n \lambda_k^s g_k^2 + T  \sum_{k=m}^n h_k^2 + T^{2(\gamma-1)} \sum_{k=m}^{n} \int_0^T f_k^2(t) \diff t.
 \label{eq:energy_2}
\end{equation}
This immediately yields
\begin{equation}
 \lim_{m,n \rightarrow \infty} \left\| \sum_{k=m}^n \partial_t u_k(t) \varphi_k \right\|_{L^2(0,T;L^2(\Omega))}  = 0.
 \label{eq:limit_2}
\end{equation}

\noindent \boxed{\emph{Step 3.}} In this step we analyze, again, the convergence of $\sum_k \partial_t u_k\varphi_k$, but on the space $L^{\infty}(0,T;L^2(\Omega))$. To accomplish this task, we will assume $g \in \mathbb{H}^{2s}(\Omega)$, $h \in L^2(\Omega)$ and $f \in L^{\infty}(0,T;L^2(\Omega))$. With this setting at hand, similar arguments to the ones that led to \eqref{eq:limit_2} allow us to conclude that
\begin{align}
 \sup_{t \in [0,T]} \sum_{k=m}^n |\partial_t u_k(t)|^2\diff t  & \lesssim T^{2(\gamma-1)}  \sum_{k=m}^n \lambda_k^{2s} g_k^2 
 \nonumber
 \\
 & + \sum_{k=m}^n h_k^2 + T^{2(\gamma-1)}  \sup_{t \in [0,T]} \sum_{k=m}^n f_k(t)^2,
 \label{eq:energy_3}
\end{align}
and thus that
\begin{equation}
 \lim_{m,n \rightarrow \infty} \sup_{t \in [0,T]} \left\| \sum_{k=m}^n \partial_t u_k(t) \varphi_k \right\|_{L^2(\Omega)}  = 0.
 \label{eq:limit_3}
\end{equation}

\noindent \boxed{\emph{Step 4.}} Collecting the results \eqref{eq:limit_1}, \eqref{eq:limit_2} and \eqref{eq:limit_3}, we conclude that 
\[
 \sum_{k\in \mathbb{N}} u_k(t) \varphi_k(x'), \qquad \sum_{k\in \mathbb{N}} \partial_t u_k(t) \varphi_k(x')
\]
converge in $L^{\infty}(0,T;\Hs)$ and $L^{\infty}(0,T;L^2(\Omega))$, respectively. Consequently, $u$, given by \eqref{eq:solution_to_fractional_wave}, belongs to $L^{\infty}(0,T;\Hs)$ and its first--order time derivative $\partial_t u \in L^{\infty}(0,T;L^2(\Omega))$. 

Since $\mathcal{L}^s$ is an isomorphism from $\Hs$ onto $\Hsd$ and $f \in L^2(0,T;L^2(\Omega))$, it is immediate that $\partial_t^{\gamma} u \in L^2(0,T;\Hsd)$. The energy estimate \eqref{eq:energy_estimate_gamma_1} follows from \eqref{eq:energy_1} and \eqref{eq:energy_2}, while \eqref{eq:energy_estimate_gamma_2} follows from \eqref{eq:energy_3}. In addition, \eqref{eq:energy_1} and \eqref{eq:energy_3} show that $u \in C([0,T];\Hs)$ and $\partial_t u \in C([0,T];L^2(\Omega))$. We thus conclude that $u$ is a weak solution to problem \eqref{eq:fractional_wave} in the sense of Definition \ref{def:weak_waveg12}. Uniqueness follows standard arguments. We leave details to the reader.
\end{proof}

\begin{remark}[initial conditions $\gamma \in (1,2)$]\rm
The estimate \eqref{eq:energy_1} reveals that $u \in C([0,T];\Hs)$ and thus that the initial condition $u(0)=g$ is well--defined. Since $u \in L^2(0,T;\Hs)$ and $\partial_t u \in L^2(0,T;L^2(\Omega))$, this result can also be obtained by using \cite[Lemma 7.3]{MR3014456}. However, to conclude that $\partial_t u \in C([0,T];L^2(\Omega))$ and thus that $\partial_t u(0)=h$ is well--defined, additional assumptions need to be made on the problem data: $g \in \mathbb{H}^{2s}(\Omega)$ and $f \in L^{\infty}(0,T;L^2(\Omega))$. This is in contrast to the case $\gamma=2$ where, to have $\partial_t u \in C([0,T];L^2(\Omega))$, weaker assumptions are needed on the problem data; see Remark \ref{rem:initial_cond_2}.
\end{remark}

\begin{remark}[weak solution $\gamma \in (1,2)$]\rm
Definition \ref{def:weak_waveg12}, of a weak solution to problem with $\gamma \in (1,2)$, is inspired in the work \cite{Sakayama}. Recently in \cite{MR3613323}, the authors have considered a different definition of weak solutions, that is in turn inspired in \cite{Sakayama} and \cite{MR3465296}; see \cite[Definition 1.1]{MR3613323}. The advantage of this approach is that it allows to show the well--posedness of problem \eqref{eq:fractional_wave} under weaker assumptions on the problem data. In particular, the authors show that $u \in C([0,T];L^2(\Omega))$; see \cite[Theorem 1.1]{MR3613323}. In contrast, in Theorem \ref{thm:wp_gamma} we assume more regularity on problem data mainly to obtain that $\partial_t u \in C([0,T];L^2(\Omega))$.
\end{remark}

\subsection{The extended fractional wave equation}
\label{sub:extended}

In this section we briefly analyze the extended problem \eqref{eq:wave_alpha_extension}--\eqref{eq:initial_cond} and derive a representation formula for its solution. To accomplish this task, let us first recall that $u(x',t) = \sum_{k \geq 1} u_k(t) \varphi_k(x')$ corresponds to the unique weak solution to problem \eqref{eq:fractional_wave}, where, for $k \in \mathbb{N}$, $u_k(t)$ solves \eqref{eq:uk} and $\varphi_k$ is defined as in \eqref{eq:eigenpairs}. Let us thus consider the formal expression 
\begin{equation}
\label{eq:exactforms}
  \ue(x,t) = \sum_{k \geq 1} u_k(t) \varphi_k(x') \psi_k(y),
\end{equation}
where, for $\alpha = 1-2s$, the functions $\psi_k$ solve
\begin{equation}
\label{eq:psik}
\begin{dcases}
  \frac{\diff^2}{\diff y^2}\psi_k(y) + \frac{\alpha}{y} \frac{\diff}{\diff y} \psi_k(y) = \lambda_k \psi_k(y) , 
& y \in (0,\infty), 
\\
\psi_k(0) = 1, & \lim_{y\rightarrow \infty} \psi_k(y) = 0.
\end{dcases}
\end{equation}
If $s = \tfrac{1}{2}$, we thus have $\psi_k(y) = \exp(-\sqrt{\lambda_k}y)$ \cite[Lemma 2.10]{CT:10}; more generally, if $s \in (0,1) \setminus \{ \tfrac{1}{2}\}$, then \cite[Proposition 2.1]{CDDS:11}
\begin{equation}
\label{eq:psik_representation}
 \psi_k(y) = c_s (\sqrt{\lambda_k}y)^s K_s(\sqrt{\lambda_k}y),
\end{equation}
where $c_s = 2^{1-s}/\Gamma(s)$ and $K_s$ denotes the modified Bessel function of the second kind. We refer the reader to \cite[Chapter 9.6]{Abra} and \cite[Chapter 7.8]{MR0435697} for a comprehensive treatment of the Bessel function $K_s$. We immediately comment the following property that the function $\psi_k$ satisfies:
 \[
  \lim_{s \rightarrow \frac{1}{2}}\psi_k(y)= \exp(-\sqrt{\lambda_k}y) \quad \forall y>0.
 \]
This property follows from the fact that 
\[
 c_{\frac{1}{2}} = \sqrt{\tfrac{2}{\pi}}, \quad K_{\frac{1}{2}}(z) = \sqrt{\tfrac{\pi}{2z}}e^{-z};
\]
the latter being a consequence of formulas (9.2.10) and (9.6.10) in \cite{Abra}.

We now notice that in view of \eqref{eq:exactforms} and \eqref{eq:psik} we conclude that
\[
  \ue(x',y,t)|_{y=0} = \sum_{k \geq 1} u_k(t) \varphi_k(x') \psi_k(0) =
  \sum_{k\geq1} u_k(t) \varphi_k(x') = u(x',t),
\]
for a.e.~$t \in (0,T)$, \ie $u = \tr \ue$, where, we recall that, $\tr$ denotes the trace operator and is defined in Section \ref{sub:CaffarelliSilvestre}.

Properties of the modified Bessel function of the second kind $K_s$ imply the following important properties of the functions $\psi_k$ that solve \eqref{eq:psik}. For $s \in (0,1)$, we have that \cite[formula (2.31)]{NOS}
\begin{equation}
 \label{asym-psi_k-p}
 \lim_{y \downarrow 0 } \frac{y^{\alpha}\psi_{k}'(y)}{d_s \lambda_k^s } = - 1,
\end{equation}
and, for $a, b \in \R^+$, $a<b$, that \cite[formula (2.33)]{NOS}
\begin{equation}
 \label{besselenergy}
 \int_{a}^{b} y^{\alpha} \left( \lambda_k \psi_k(y)^2 + \psi_k'(y)^2 \right) \diff y
 = \left. y^{\alpha} \psi_k(y)\psi_k'(y) \right|_{a}^{b}.
\end{equation}

We thus combine the definition \eqref{eq:conormal_derivative} with \eqref{eq:exactforms} and \eqref{asym-psi_k-p} to arrive at
\[
 d_s \tr \partial_t^{\gamma} \ue + \partial_{\nu}^{\alpha} \ue =  d_s \tr \partial_t^{\gamma} \ue - \lim_{y \downarrow 0} y^{\alpha} \ue_y = d_s  \sum_{k=1}^{\infty} \varphi_{k} \left( \partial_t^{\gamma} u_k + \lambda_k^s u_k  \right) = d_s f,
\]
where the last equality is a consequence of the fact that $u_k$ solves \eqref{eq:uk}. These results thus show that $\ue$, given by \eqref{eq:exactforms}, satisfies the boundary condition of problem \eqref{eq:wave_alpha_extension}.

Finally, on the basis of the asymptotic estimate \eqref{asym-psi_k-p}, \eqref{besselenergy} implies that
\begin{align}
\| \nabla \ue(t) \|^2_{L^2(y^{\alpha},\C)} & \lesssim \sum_{k=1}^{\infty} u_k(t)^2 \int_{0}^{\infty} 
y^{\alpha} \left( \lambda_k \psi_k(y)^2 + \psi_k'(y)^2 \right) \diff y
\nonumber
\\
& = d_s \sum_{k=1}^{\infty} \lambda_k^s u_k(t)^2 = d_s \| u(t)\|^2_{\Hs},
\label{norm=} 
\end{align}
for a.e.~$t \in (0,T)$. 

These arguments show that given a function $u \in L^{\infty}(0,T;\Hs)$ there exists $\ue \in L^{\infty}(0,T;\HL(y^{\alpha},\C))$ such that $u = \tr \ue$. Consequently, $\tr$ is such that
\[
 L^{\infty}(0,T;\Hs)\subseteq \tr \left( L^{\infty}(0,T;\HL(y^{\alpha},\C) \right),
\]
\ie $\tr: L^{\infty}(0,T;\HL(y^{\alpha},\C)) \rightarrow L^{\infty}(0,T;\Hs)$ is surjective \cite[Proposition 2.1]{CDDS:11}. Similar arguments to the ones developed in the proof of \cite[Proposition 2.1]{CDDS:11} reveal the inverse inclusion. We thus have that $\tr$ is a bounded linear map such that
\[
L^{\infty}(0,T;\Hs) = \tr \left( L^{\infty}(0,T;\HL(y^{\alpha},\C) \right).
\]

As it will be necessary in what follows we define $a: \HL(y^{\alpha},\C) \times \HL(y^{\alpha},\C) \rightarrow \mathbb{R}$
\begin{equation}
\label{a}
a(w,\phi) :=   \frac{1}{d_s}\int_{\C} y^{\alpha} \left( \mathbf{A}(x) \nabla w \cdot \nabla \phi + c(x') w \phi \right) \diff x. 
\end{equation}

The previous formal considerations, on the basis of the results of Section~\ref{sub:existunique}, suggest to consider the following notion of weak solution for problem \eqref{eq:wave_alpha_extension}--\eqref{eq:initial_cond}.

\begin{definition}[extended weak solution]
\label{def:weak_wave}
We call $\ue \in L^{\infty}(0,T;\HL(y^{\alpha},\C))$, with $\tr \partial_t \ue \in L^{\infty}(0,T;L^2(\Omega))$ and $\tr \partial_{t}^{\gamma} \ue \in L^2(0,T;\Hsd)$, a weak solution of problem \eqref{eq:wave_alpha_extension}--\eqref{eq:initial_cond} if $\tr \ue(0) = g$, $\partial_t \tr \ue(0) = h$ and, for a.e. $t \in (0,T)$,
\begin{equation}
\label{eq:weak_wave}
  \langle \tr \partial_t^{\gamma} \ue , \tr \phi \rangle + a(\ue,\phi) = \langle f , \tr \phi \rangle \quad \forall \phi \in \HL(y^\alpha,\C),
\end{equation}
where $\langle \cdot , \cdot \rangle$ denotes the the duality pairing between $\Hs$ and $\Hsd$ and the bilinear form $a$ is defined as in \eqref{a}.
\end{definition}

\begin{remark}[dynamic boundary condition] \rm
Problem \eqref{eq:weak_wave} is an elliptic problem with a dynamic boundary condition: $\partial_{\nu}^{\alpha} \ue = d_s(f - \tr  \partial_t^{\gamma} \ue)$ on $\Omega \times \{0\}$.
\end{remark}

We have the following localization result \cite{MR2600998,MR3393253,CT:10,CS:07,CDDS:11,MR2737788,MR2954615,ST:10,MR3192423}.

\begin{theorem}[Caffarelli--Silvestre extension]
\label{thm:CS_NOT_ST}
Let $\gamma \in (1,2]$ and $s \in (0,1)$. If $f$, $g$ and $h$ are as in Theorem~\ref{thm:wp_2} (for $\gamma = 2$) or as in Theorem~\ref{thm:wp_gamma} (for $\gamma<2$)
then the unique weak solution of problem \eqref{eq:fractional_wave}, in the sense of Definition \ref{def:weak_2} or Definition \ref{def:weak_waveg12}, respectively, satisfies that $u = \tr \ue$, where $\ue$ denotes the unique weak solution to problem \eqref{eq:wave_alpha_extension} in the sense of Definition \ref{def:weak_wave}.
\end{theorem}

We now present the well--posedness of problem \eqref{eq:wave_alpha_extension}--\eqref{eq:initial_cond} together with energy estimates for its solution.

\begin{theorem}[well--posedness for $\gamma = 2$] 
\label{thm:energy_wave_2}
Given $s \in (0,1)$, $\gamma = 2$, $f \in L^2(0,T;L^2(\Omega))$, $h \in \Hs$ and $g \in L^2(\Omega)$, then problem problem \eqref{eq:wave_alpha_extension}--\eqref{eq:initial_cond} has a unique weak solution in the sense of Definition \ref{def:weak_wave}. In addition,
\end{theorem}
\begin{multline}\label{eq:energy_estimate_wave_2}
\| \nabla \ue \|_{L^{\infty}(0,T;L^2(y^{\alpha},\C))} + \| \tr \partial_t \ue \|_{L^{\infty}(0,T;L^2(\Omega))} 
\\
\lesssim \| f \|_{L^2(0,T;L^2(\Omega))} + \| g \|_{\Hs} + \| h \|_{L^2(\Omega)},
\end{multline}
where the hidden constant is independent of the problem data.
\begin{proof}
The proof is based  on the so--called Galerkin method \cite{Evans,Lions,MR3014456}. For brevity, we leave the details to the reader. 
\end{proof}

\begin{theorem}[well--posedness for $\gamma \in (1,2)$]
\label{thm:energy_gamma_gamma}
Given $s \in (0,1)$, $\gamma \in (1,2)$, $f \in L^{\infty}(0,T;L^2(\Omega))$, $g \in \mathbb{H}^{2s}(\Omega)$, and $h \in L^2(\Omega)$, then problem \eqref{eq:wave_alpha_extension}--\eqref{eq:initial_cond} has a unique weak solution in the sense of Definition \ref{def:weak_wave}, which is given by \eqref{eq:exactforms}. In addition,
\begin{multline}
\label{eq:energy_estimate_gamma_wave_1}
\| \nabla \ue \|_{L^{\infty}(0,T;L^2(y^{\alpha},\C))} + \| \tr \partial_t \ue\|_{L^{2}(0,T;L^2(\Omega))} 
\\
\lesssim 
\| f \|_{L^2(0,T;L^2(\Omega))} + \| g \|_{\Hs} + \| h \|_{L^2(\Omega)},
\end{multline}
and
\begin{equation}\label{eq:energy_estimate_gamma_wave_2}
\| \tr \partial_t \ue \|_{L^{\infty}(0,T;L^2(\Omega))} \lesssim \| f \|_{L^{\infty}(0,T;L^2(\Omega))} + \| g \|_{\mathbb{H}^{2s}(\Omega)} + \| h \|_{L^2(\Omega)}.
\end{equation}
The hidden constants, in both inequalities, are independent of the problem data.
\end{theorem}
\begin{proof}
Existence and uniqueness of a weak solution $\ue$ to problem \eqref{eq:wave_alpha_extension}--\eqref{eq:initial_cond} follows from, on the basis of \eqref{eq:exactforms} and \eqref{eq:psik}, the arguments elaborated in the proof of Theorem \ref{thm:wp_gamma}. These arguments, together with \eqref{norm=}, also allow us to bound $\nabla \ue$ in $L^{\infty}(0,T;L^2(y^{\alpha},\C))$ by the terms that appear on the right--hand side of \eqref{eq:energy_estimate_gamma_wave_1}. The control of the remaining term $\tr \partial_t \ue$ follows from \eqref{eq:energy_estimate_gamma_1} and an application of Theorem \ref{thm:wp_gamma}. The estimate of $\tr \partial_t \ue$ in $L^{\infty}(0,T;L^2(\Omega))$ follows from \eqref{eq:energy_estimate_gamma_2}. This concludes the proof.
\end{proof}

\section{Regularity}
\label{sub:Regularity}

Having obtained conditions that guarantee the existence and uniqueness of solutions for problems \eqref{eq:fractional_wave} and \eqref{eq:wave_alpha_extension}--\eqref{eq:initial_cond} we now study their regularity properties. 

\subsection{Time regularity}
\label{sub:time_regularity}

In \cite{MR2200938}, the authors introduce and analyze a numerical technique based on the Crank--Nicolson scheme to discretize the left--sided Caputo fractional derivative of order $\gamma \in (1,2)$. They obtain stability properties \cite[Theorem 3.2]{MR2200938} and, under the time--regularity assumption that
\begin{equation}
\label{eq:u_C3}
 u \in C^{3}[0,T],
\end{equation}
they derive error estimates for the proposed scheme \cite[Theorem 3.3]{MR2200938}. This solution technique was later used and extended to the design of numerical techniques for different equations involving a left--sided Caputo fractional derivative of order $\gamma \in (1,2)$; see, for instance, \cite{MR3549087,MR2651618,MR3471180,MR2870026,MR3109799,MR2970754,MR3304331}. In all these references it is assumed that the solution $u$ of the involved equation satisfies \eqref{eq:u_C3}. However, \eqref{eq:u_C3} is inconsistent with \eqref{eq:formal_series}. In fact, let us assume that $f \equiv 0$ and $h \equiv 0$ in \eqref{eq:fractional_wave}. Thus, on the basis of \eqref{eq:sol_rep_GH}, properties of the Mittag--Leffler function show that, if $g \neq 0$, the solution $u$ to problem \eqref{eq:fractional_wave} satisfies the following asymptotic estimate:
\[
 u(x',t) = G_{\gamma}g(x) = \left( 1 - \frac{t^{\gamma}}{\Gamma(1+\gamma)} \mathcal{L}^s + \mathcal{O}(t^{2\gamma}) \right) g(x'), \quad t \downarrow 0.
\]
If this is the case, we then expect that $\partial_t^2 u(x',t) \approx t^{\gamma-2} \mathcal{L}^s g(x')$ and that $\partial_t^3 u(x',t) \approx t^{\gamma-3} \mathcal{L}^s g(x')$. Consequently, since $\gamma \in (1,2)$, we have that
\emph{the second derivatives of $u$ with respect to the variable $t$ are unbounded as $t \downarrow 0$.} Unfortunately, we also have that have $\partial_t^3 u \notin L^2(0,T; \Hsd))$. However, this asymptotic also suggests that if $0<\varepsilon<T$
\begin{equation}
\label{eq:weighted_estimates}
\int_0^{\varepsilon} t^\rho \| \partial_t^3 u(\cdot,t) \|_{\Hsd}^2 \diff t \lesssim \| g \|_{\Hs}^2 \int_0^{\varepsilon} t^{\rho+2\gamma-6} \diff t
\end{equation}
is finite provided $\rho > 5 -2\gamma$, \ie $\partial_t^3 u \in L^2( t^\rho,(0,T);\Hsd)$. The justification of this heuristic is the content of Theorem \ref{thm:time_regularity} below. In order to derive such a result, we first establish some suitable bounds for the solution operators $G_{\gamma}$ and $H_{\gamma}$ that are defined in \eqref{eq:sol_rep_GH}. We refer the reader to \cite{M:10,MR1225703} for similar bounds but when $\gamma \in (0,1)$.

\begin{lemma}[estimates for $G_{\gamma}$ and $H_{\gamma}$]
\label{lemma:G_gamma}
Let $\gamma \in (1,2)$, $s \in (0,1)$ and $r \in \mathbb{R}$. If $q$ denotes a positive integer, then, for a.e. $t \in (0,T]$, we have the following estimates
\begin{equation}
\label{eq:G_gamma}
  \| \partial_t^q G_{\gamma}(t) w \|_{\mathbb{H}^r(\Omega)} + \| \partial_t^{q+1}  H_{\gamma}(t) w \|_{\mathbb{H}^r(\Omega)}  \lesssim t^{\gamma -q} \| w \|_{\mathbb{H}^{r+2s}(\Omega)}, 
 \end{equation}
 and
 \begin{equation}
\label{eq:H_gamma}
  \| \partial_t^{q+1}  H_{\gamma}(t) w \|_{\mathbb{H}^r(\Omega)} \lesssim t^{\gamma/2 -q} \| w \|_{\mathbb{H}^{r+s}(\Omega)}.
 \end{equation}
The hidden constants, in both inequalities, are independent of $t$, $w$, $s$, $r$ and $q$.
\end{lemma}
\begin{proof} 
Invoke the definition of $G_{\gamma}$, given in \eqref{eq:sol_rep_GH}, the definition of the space $\mathbb{H}^{r}(\Omega)$, given in \eqref{def:Hs}, the fact that $\{ \varphi_{\ell}\}_{\ell \in \mathbb{N}}$ is an orthonormal basis of $L^2(\Omega)$ and the differentiation formula \eqref{eq:G_derivatives} to conclude that
\begin{align}
  \|\partial_t^q G_{\gamma}(t) w \|^2_{\mathbb{H}^r(\Omega)} & = \sum_{k = 1}^{\infty} \lambda_k^r w_k^2 (\partial_t^q E_{\gamma,1}(-\lambda_k^s t^{\gamma}))^2
  \nonumber
   \\
  & = t^{2(\gamma-q)}\sum_{k = 1}^{\infty} \lambda_k^{r+2s} w_k^2 E^2_{\gamma,\gamma-q+1}(-\lambda_k^s t^{\gamma}). 
  \label{eq:aux_sum}
\end{align}
Consequently, employing the estimate \eqref{ML_estimate} we arrive at
\begin{equation}
  \|\partial_t^q G_{\gamma}(t) w \|^2_{\mathbb{H}^r(\Omega)} \lesssim t^{2(\gamma-q)}\sum_{k = 1}^{\infty} \lambda_k^{r+2s} w_k^2  = t^{2(\gamma-q)} \| w \|^2_{\mathbb{H}^{r+2s}(\Omega)}, 
\end{equation}
which yields the first part of \eqref{eq:G_gamma}. The derivation of the second part follows similar arguments: see \eqref{eq:sol_rep_GH} and the differentiation formulas \eqref{eq:G_derivatives2} and \eqref{eq:G_derivatives}.

The estimate \eqref{eq:H_gamma} follows similar arguments. In fact,
\begin{align*}
  \| \partial_t^{q+1} H_\gamma(t) w \|_{\mathbb{H}^r(\Omega)}^2 &= t^{\gamma-2q} \sum_{k=1}^\infty \lambda_k^{r+s} w_k^2 \left[ \lambda_k^s t^\gamma E_{\gamma,\gamma-q+1}(-\lambda_k t^\gamma)^2 \right] \\
  & \lesssim t^{\gamma - 2q} \sum_{k=1}^\infty \lambda_k^{r+s} w_k^2 \frac{ \lambda_k^s t^\gamma}{(1+ \lambda_k t^\gamma)^2}
  \lesssim t^{\gamma - 2q} \sum_{k=1}^\infty \lambda_k^{r+s} w_k^2.
\end{align*}

This concludes the proof.
\end{proof}

To present the following result we define
\begin{equation}
\label{A}
\mathcal{A}(g,h,f) =  \| g \|_{\Hs} + \| h\|_{L^2(\Omega)} + \| f\|_{H^2(0,T;\Hsd)}.
\end{equation}

\begin{theorem}[time regularity: $\gamma \in (1,2)$]
\label{thm:time_regularity}
Assume that $g \in \Hs$, $h \in L^2(\Omega)$ and $f \in H^2(0,T;\Hsd)$. Then, the solution $u$ to problem \eqref{eq:fractional_wave} satisfies
\begin{equation}
\label{eq:reg_sigma}
 \|t^{\rho/2} \partial_t^3 u \|_{L^2(0,T;\Hsd)} \lesssim \mathcal{A}(g,h,f) ,
\end{equation}
where $\rho > 5 - 2 \gamma$. The hidden constant is independent of $t$ and the problem data but blows up as $\gamma \downarrow 1$.
 \end{theorem}
 
\begin{proof}
We proceed in three steps and apply the superposition principle.

\noindent \boxed{\emph{Step 1. Case $h \equiv 0$ and $f \equiv 0$.}} In this case, the solution to problem \eqref{eq:fractional_wave} is $u(x',t) = G_{\gamma}(t)g(x')$. An application of the estimate \eqref{eq:G_gamma} of Lemma \ref{lemma:G_gamma} reveals that 
\[
 \| \partial_t^3 u \|_{\Hsd} \lesssim t^{\gamma-3}\| g\|_{\Hs} 
\]
for $t \in (0,T]$, whence \eqref{eq:reg_sigma} follows.

\noindent \boxed{\emph{Step 2. Case $g \equiv 0$ and $f \equiv 0$.}} If this is the case, the solution to problem \eqref{eq:fractional_wave} is $u(x',t) = H_{\gamma}(t)g(x')$, where $H_{\gamma}$ is defined as in \eqref{eq:sol_rep_GH}. We thus apply the estimate \eqref{eq:H_gamma} of Lemma \ref{lemma:G_gamma} to obtain, for $t \in (0,T]$, that $\| \partial_t^3 u \|_{\Hsd} \lesssim t^{\gamma/2-2}\| h \|_{L^2(\Omega)}$. This immediately yields \eqref{eq:reg_sigma}.

\noindent \boxed{\emph{Step 3. Case $g \equiv 0$ and $h \equiv 0$.}}  In this case, we have that $u = \sum_{k} u_k \varphi_k$, where $u_k$ is as in formula \eqref{eq:u_k_rs} with $g_k = h_k = 0$:
\[
 u_k(t) = \int_{0}^t (t-r)^{\gamma-1}E_{\gamma,\gamma}(-\lambda_k^s(t-r)^{\gamma})f_k(r)\diff r, \qquad k \in \mathbb{N}.
\]
The first--order derivative of $u_k$ is given in \eqref{eq:du_k_rs}: $\partial_t u_k(t) = \int_{0}^t (t-r)^{\gamma-2} E_{\gamma,\gamma-1}(-\lambda_k^s(t-r)^{\gamma} ) f_k(r) \diff r$; notice that $g_k = h_k = 0$. On the basis of this formula, a simple change of variable and differentiation allow us to conclude that
\[
\partial_t^2 u_k(t) = t^{\gamma-2} E_{\gamma,\gamma-1}(-\lambda_k^s t^{\gamma}) f_k(0) + \int_{0}^t r^{\gamma-2} E_{\gamma,\gamma-1}(-\lambda_k^s r^{\gamma})\partial_t f_k(t-r) \diff r.
\]
Differentiating once more, and using formula \eqref{eq:derivative_t_ML}, that yields
\[
 \partial_t \left ( t^{\gamma-2}E_{\gamma,\gamma-1}(-\lambda_k^s t^{\gamma}) \right) = t^{\gamma-3}E_{\gamma,\gamma-2}(-\lambda_k^s t^{\gamma}),
\]
we obtain that
\begin{multline}
\partial_t^3 u_k(t) = t^{\gamma-3} E_{\gamma,\gamma-2}(-\lambda_k^s t^{\gamma})f_k(0)   + t^{\gamma-2} E_{\gamma,\gamma-1}(-\lambda_k^s t^{\gamma}) \partial_t f_k(0) 
\\
+\int_0^t r^{\gamma-2}E_{\gamma,\gamma-1}(-\lambda_k^s r^{\gamma}) \partial_t^2 f_k(t-r) \diff r. 
\label{eq:3rd_derivative}
\end{multline}
Since $\rho > 5-2\gamma$ yields $\int_{0}^T r^{\rho + 2\gamma - 6} \diff r < \infty$, the first and second term on the right--hand side of the previous expression lead to \eqref{eq:reg_sigma}. To estimate the third term we first use that $\rho>0$, invoke then estimate \eqref{ML_estimate} and then Lemma \ref{le:continuity} to conclude that 
\begin{align*}
  \int_0^T t^\rho \left( \int_0^t r^{\gamma-2} E_{\gamma,\gamma-1}(-\lambda_k^s r^\gamma) \partial_t^2 f_k(t-r) \diff r \right)^2 \diff t &\lesssim
  \left\| r^{\gamma - 2} \star \partial_t^2 f_k \right\|_{L^2(0,T)}^2 \\
  &\lesssim \| \partial_t^2 f_k \|_{L^2(0,T)}^2,
\end{align*}
where we also used that, since $\gamma \in (1,2)$, $r^{\gamma-2} \in L^1(0,T)$ and $\| r^{\gamma -2 } \|_{L^1(0,T)}= \tfrac{1}{\gamma-1}T^{\gamma-1}$. This concludes the proof.
\end{proof}

On the basis of Theorem \ref{thm:CS_NOT_ST} we immediately conclude the following result.

\begin{corollary}[time regularity for the extension: $\gamma \in (1,2)$]
\label{thm:time_regularity_extension}
Assume that $g \in \Hs$, $h \in L^2(\Omega)$ and $f \in H^2(0,T;\Hsd)$. Then, the solution $\ue$ to problem \eqref{eq:wave_alpha_extension}--\eqref{eq:initial_cond} satisfies
\begin{equation}
\label{eq:reg_sigma_extension}
 \|t^{\rho/2} \tr \partial_t^3 \ue \|_{L^2(0,T;\Hsd)} \lesssim \mathcal{A}(g,h,f) ,
\end{equation}
where $\rho > 5 - 2 \gamma$. The hidden constant is independent of $t$ and the problem data, but blows up as $\gamma \downarrow 1$.
 \end{corollary}

\subsection{Space regularity}
\label{sub:space_regularity}

We now proceed to analyze the space regularity properties of the solution $\ue$ of problem \eqref{eq:wave_alpha_extension}--\eqref{eq:initial_cond}. To accomplish this task, we define the weight
\begin{equation}
\label{eq:weight}
\omega_{\beta,\theta}(y) = y^{\beta}e^{\theta y}, \qquad 0 \leq \theta < 2 \sqrt{\lambda_1},
\end{equation}
where $\beta \in \R $ will be specified later. With this weight at hand, we define the norm
\begin{equation}
 \label{eq:weighted_norm}
 \| v \|_{L^2(\omega_{\beta,\theta},\C)} 
:= \left( \int_0^{\infty} \int_{\Omega} \omega_{\beta,\theta}(y) 
          |v(x',y)|^2 \diff x' \diff y
 \right)^{\frac{1}{2}}.
\end{equation}

In view of formulas \eqref{eq:exactforms} and \eqref{eq:psik_representation} we observe that, in order to derive regularity properties of $\ue$ it is essential to bound certain weighted integrals of the derivatives of the function $\psi(z):= c_s z^ s K_s(z)$. To accomplish this task, we define, for $\beta$, $\delta \in \mathbb{R}$, $\ell \in \mathbb N$, and $\lambda > 0$
\begin{equation}
\label{eq:defofbigPhi}
\Phi(\delta,\theta,\lambda) 
= 
\int_{ 0 }^{ \infty } z^{ \delta }e^{\theta z/\sqrt{\lambda}} \left| \psi(z) \right|^2 \diff z
\end{equation}
and
\begin{equation}
\label{eq:defofbigPsi}
\Psi_\ell(\beta,\theta,\lambda) 
= \int_{ 0 }^{ \infty } z^{\beta + 2\ell}e^{\theta z/\sqrt{\lambda}} \left| \frac{\diff^\ell}{\diff z^\ell} \psi(z) \right|^2 \diff z.
\end{equation}
The parameter $\theta$ is such that \eqref{eq:weight} holds.

The integrals $\Phi(\delta,\theta,\lambda)$ and $\Psi_\ell(\beta,\theta,\lambda)$ are bounded as follows.

\begin{proposition}[bounds on $\Phi$ and $\Psi_\ell$]
\label{lem:bound_Psi_finite_interval}
Let $\delta > -1$, $\beta > -1 - 4s$, $\ell \in \mathbb{N}$. If $\theta$ is such that $0 \leq \theta < 2 \sqrt{\lambda_1}$ and $\lambda \geq \lambda_1$, then
\begin{equation}
\label{eq:bound_on_Phi}
\Phi(\delta,\theta,\lambda) \lesssim 1,
\end{equation}
where the hidden constant is independent of $\lambda$. In addition, there exists $\kappa > 1$ such that 
\begin{equation}
\label{eq:bound_on_Psi}
\Psi_\ell(\beta,\theta,\lambda) \lesssim \kappa^{2\ell} (\ell!)^2,
\end{equation}
where the hidden constant is independent of $\ell$ and $\lambda$.
\end{proposition}
\begin{proof}
See \cite[Lemma 4.6]{BMNOSS:17}
\end{proof}

On the basis of Proposition~\ref{lem:bound_Psi_finite_interval} we can give pointwise, in time, bounds for $\ue$.

\begin{theorem}[pointwise bounds]
\label{thm:pointwisebounds}
Le $\ue$ solve \eqref{eq:wave_alpha_extension}--\eqref{eq:initial_cond} for $s\in (0,1)$ and $\gamma \in (1,2]$. Let $0 \leq \sigma < s$ and $0 \leq \nu < 1+s $. Then, there exists $\kappa > 1$ such that the following bounds hold:
\begin{align}
\label{eq:partial_y_l+1}
  \| \partial_y^{\ell+1} \ue(\cdot,t) \|_{L^2(\omega_{\alpha+2\ell-2\sigma,\theta}, \C)}^2 &\lesssim \kappa^{2(\ell+1)} (\ell+1)!^2 \| u(\cdot,t) \|_{\mathbb{H}^{\sigma+s}(\Omega)}^2, \\
\label{eq:nablaxpartial_y_l+1}
  \| \nabla_{x'} \partial_y^{\ell+1} \ue(\cdot,t) \|_{L^2(\omega_{\alpha+2(\ell+1)-2\nu,\theta},\C)}^2 &\lesssim \kappa^{2(\ell+1)} (\ell+1)!^2 \| u(\cdot,t) \|_{\mathbb{H}^{\nu+s}(\Omega)}^2, \\
\label{eq:deltaxpartial_y_l+1}
   \|  \mathcal{L}_{x'}  \partial_y^{\ell+1} \ue(\cdot,t) \|^2_{L^2(\omega_{\alpha+2(\ell+1)-2\nu,\theta},\C)} 
   &\lesssim \kappa^{2(\ell+1)} (\ell+1 )!^2 \| u(\cdot,t)\|_{\mathbb{H}^{1 + \nu + s}(\Omega)}^2.
\end{align}
In all inequalities the hidden constants are independent of $\ue$ and $\ell$.
\end{theorem}
\begin{proof}
We begin with the proof of \eqref{eq:partial_y_l+1}. We invoke the representation formula \eqref{eq:exactforms} and use the fact that $\{ \varphi_k \}_{k=1}^{\infty}$ is an orthonormal basis of $L^2(\Omega)$ to arrive at
\[
\|\partial_{y}^{\ell+1} \ue(\cdot,t) \|^2_{L^2(\omega_{\alpha + 2 \ell -2\sigma,\theta},\C)}  
= \sum_{k=1}^{\infty} u_k^2(t)
  \int_0^\infty y^{\alpha + 2 \ell -2\sigma}e^{\theta y}
\left|\frac{\diff^{\ell+1} }{\diff y^{\ell+1}} \psi_k (y) \right|^2 \diff y.
\]
Consider the change of variable $z = \sqrt{\lambda_k} y$ on the previous integral. Thus, in view of definition \eqref{eq:defofbigPsi}, we can apply the estimate \eqref{eq:bound_on_Psi} with $\beta = \alpha - 2\sigma -2 = 1-2s -2\sigma -2 > -1-4s$ to conclude that
\begin{align*}
\nonumber
\|\partial_{y}^{\ell+1} \ue(\cdot,t) \|^2_{L^2(\omega_{\alpha + 2 \ell -2\sigma,\theta},\C)}  
& = 
\sum_{k=1}^{\infty}
\lambda_k^{ (\ell + 1) - 
\left(\frac{\alpha+2\ell-2\sigma}{2}\right) -\frac12}
u_k^2(t) 
\Psi_{\ell+1}(\alpha-2\sigma-2,\theta,\lambda_k)
\\
& \lesssim (\ell+1)!^2 \kappa^{2(\ell+1)}\sum_{k=1}^{\infty} \lambda_k^{\sigma + s} u_k^2(t),
\end{align*}
as we intended to show.

Similar arguments reveal that
\begin{align*}
  & \|  \nabla_{x'} \partial_y^{\ell+1} \ue( \cdot,t) \|^2_{L^2(\omega_{\alpha+2(\ell+1)-2\nu,\theta},\C)} 
  \\
  &= \sum_{k=1}^{\infty} u_k^2(t) \lambda_k \int_0^\infty y^{\alpha+2(\ell+1)-2\nu} 
 e^{\theta y}\left|\frac{\diff^{\ell+1} }{\diff y^{\ell+1}} \psi_k (y) \right|^2 \diff y
  \\
  &= \sum_{k=1}^\infty u_k^2(t) \lambda_k^{1 + (\ell+1) - \left( \frac{\alpha+2(\ell+1)-2\nu}{2} \right) -\frac12} \Psi_{\ell+1}(\alpha-2\nu,\theta,\lambda_k),
\end{align*}
Since $\alpha-2\nu > 1-2s - 2 -2s = -1-4s$ we can thus apply the estimate \eqref{eq:bound_on_Psi} with 
$\beta = \alpha - 2\nu$ to obtain that
\begin{equation*}
   \|  \nabla_{x'}  \partial_y^{\ell+1} \ue \|^2_{L^2(\omega_{\alpha+2(\ell+1)-2\nu,\theta},\C)} 
   \lesssim (\ell+1 )!^2 \kappa^{2(\ell+1)} \sum_{k=1}^{\infty} \lambda_k^{\nu + s} u_k^2(t).
\end{equation*}

Finally, applying the same arguments that led to \eqref{eq:partial_y_l+1} and \eqref{eq:nablaxpartial_y_l+1} we obtain that
\begin{equation*}
   \|  \mathcal{L}_{x'}  \partial_y^{\ell+1} \ue \|^2_{L^2(\omega_{\alpha+2(\ell+1)-2\nu,\theta},\C)} 
   \lesssim \kappa^{2(\ell+1)} (\ell+1 )!^2 \sum_{k=1}^{\infty} \lambda_k^{1 + \nu + s} u_k^2(t).
\end{equation*}

This concludes the proof.
\end{proof}

As an application of this result we can obtain spatial regularity for $\ue$. The results below show the spatial analyticity of $\ue$ with respect to the extended variable $y \in (0,\infty)$. We obtain that $\ue$ belongs to countably normed, power--exponentially weighted Bochner spaces of analytic functions with respect to $y$, taking values in spaces $\mathbb{H}^r(\Omega)$.

Let us first focus on the case $\gamma \in (1,2)$.

\begin{corollary}[space regularity, $\gamma \in (1,2)$]
\label{TH:regularity}
Let $\ue$ solve \eqref{eq:wave_alpha_extension}--\eqref{eq:initial_cond} for $s \in (0,1)$ and $\gamma \in (1,2)$. Let $0 \leq \sigma < s$ and $0 \leq \nu < 1+s $. Let $0 < \mu \ll 1$ be arbitrary. Then, there exists $\kappa > 1$ such that we have that
\begin{multline}
\| \partial_{y}^{\ell +1 } \ue \|_{L^2(0,T;L^2(\omega_{\alpha + 2 \ell -2\sigma,\theta},\C))}^2  
\lesssim  \kappa^{2(\ell+1)}(\ell+1)!^2 \big(
\| g \|^2_{\mathbb{H}^{\sigma + s}(\Omega)} + 
\| h \|^2_{\mathbb{H}^{\sigma}(\Omega)} 
\\
+ 
\| f \|_{L^2(0,T;\mathbb{H}^{\sigma - s + 2 \mu s}(\Omega))}^2 \big),
\label{eq:reg_in_y_gamma_1} 
\end{multline} 
\begin{multline}
\| \nabla_{x'} \partial_{y}^{\ell +1 } \ue \|_{L^2(0,T;L^2(\omega_{\alpha + 2 (\ell+1) -2\nu,\theta},\C))}^2  
\lesssim  \kappa^{2(\ell+1)}(\ell+1)!^2 \big(
\| g \|^2_{\mathbb{H}^{\nu + s}(\Omega)} + 
\| h \|^2_{\mathbb{H}^{\nu}(\Omega)} 
\\
+ 
\| f \|_{L^2(0,T;\mathbb{H}^{\nu - s + 2 \mu s}(\Omega))}^2 \big).
\label{eq:reg_in_y_gamma_2} 
\end{multline} 
and
\begin{multline}
\| \mathcal{L}_{x'} \partial_{y}^{\ell +1 } \ue \|_{L^2(0,T;L^2(\omega_{\alpha + 2 (\ell+1) -2\nu,\theta},\C))}^2  
\lesssim  \kappa^{2(\ell+1)}(\ell+1)!^2 \big(
\| g \|^2_{\mathbb{H}^{1+\nu + s}(\Omega)} \\ + 
\| h \|^2_{\mathbb{H}^{1+\nu}(\Omega)} 
+ 
\| f \|_{L^2(0,T;\mathbb{H}^{1+\nu - s + 2 \mu s}(\Omega))}^2 \big).
\label{eq:reg_in_y_gamma_3} 
\end{multline} 
\end{corollary} 
\begin{proof}
Estimates \eqref{eq:partial_y_l+1}--\eqref{eq:deltaxpartial_y_l+1} reveal that it suffices to bound $\| u_k\|_{L^2(0,T)}$. To obtain such a bound, we recall that formula \eqref{eq:u_k_rs} reads
\[
 u_k(t)= E_{\gamma,1} (-\lambda_{k}^s t^{\gamma}) g_k + t E_{\gamma,2} (-\lambda_{k}^s t^{\gamma}) h_k + \int_{0}^t (t-r)^{\gamma-1}E_{\gamma,\gamma}(-\lambda_k^s(t-r)^{\gamma})f_k(r)\diff r.
\]
The control of the first and second terms, on the right--hand side of the previous expression, follow from the estimate \eqref{ML_estimate}. In fact, we have that
\begin{equation}
\label{eq:first_term}
\| E_{\gamma,1} (-\lambda_{k}^s t^{\gamma}) g_k \|^2_{L^2(0,T)} \lesssim T g_k^2, \quad k \in \mathbb{N}
\end{equation}
and that
\begin{equation}
\label{eq:second_term}
  \begin{aligned}
\|  t E_{\gamma,2} (-\lambda_{k}^s t^{\gamma}) h_k \|^2_{L^2(0,T)} &\lesssim \lambda_k^{-s} h_k^2 \int_0^T t^{2-\gamma} \frac{ \lambda_k^s t^\gamma}{(1+\lambda_k^s t^\gamma)^2} \diff t \\
&\lesssim \lambda_k^{-s} h_k^2\int_0^T t^{2-\gamma} \diff t \lesssim T^{3-\gamma} \lambda_k^{-s }h_k^2, \quad k \in \mathbb{N}.
  \end{aligned}
\end{equation}
To control the third term, we invoke the estimate of Lemma \ref{le:continuity} and conclude that
\begin{multline*}
\left\| \int_{0}^t (t-r)^{\gamma-1}E_{\gamma,\gamma}(-\lambda_k^s(t-r)^{\gamma})f_k(r)\diff r \right\|_{L^2(0,T)} \leq \\ \|r^{\gamma-1}E_{\gamma,\gamma}( -\lambda_k^s r^{\gamma} ) \|_{L^1(0,T)} \| f_k \|_{L^2(0,T)}.
\end{multline*}
It suffices to bound $\|r^{\gamma-1}E_{\gamma,\gamma}( -\lambda_k^s r^{\gamma} ) \|_{L^1(0,T)}$. To accomplish this task, we utilize, again, the estimate \eqref{ML_estimate}. This yields 
\begin{equation}
\label{eq:log-term}
\begin{aligned}
  \|r^{\gamma-1}E_{\gamma,\gamma}( -\lambda_k^s r^{\gamma} ) \|_{L^1(0,T)} &\lesssim \int_0^T \frac{r^{\gamma-1}}{1+ \lambda_k^s r^\gamma} \diff r
  = \gamma^{-1}\lambda_k^{-s}\int_0^{\lambda_k^s T^\gamma} \frac{\diff \xi}{1+\xi} \\ &= \gamma^{-1}\lambda_k^{-s}\log (1+\lambda_k^s T^{\gamma}).
\end{aligned}
\end{equation}
A collection of the derived estimates \eqref{eq:first_term}--\eqref{eq:log-term} reveals that
\begin{equation}
\label{eq:uk_L2}
 \| u_k \|^2_{L^2(0,T)} \lesssim T g_k^2 + T^{3-\gamma} \lambda_k^{-s} h_k^2 + \lambda_k^{-2s}\log (1+\lambda_k^s T^{\gamma})^2 \| f_k\|_{L^2(0,T)}^2.
\end{equation}
Consequently, on the basis of \eqref{eq:partial_y_l+1}, the previous estimate \eqref{eq:uk_L2} yields
\begin{multline*}
\int_0^T \|\partial_{y}^{\ell+1} \ue(\cdot,t) \|^2_{L^2(\omega_{\alpha + 2 \ell -2\sigma,\theta},\C)} \diff t \lesssim (\ell+1)!^2 \kappa^{2(\ell+1)}\bigg( T \| g \|^2_{\mathbb{H}^{\sigma+s}(\Omega)} \\
+ T^{3-\gamma} \| h \|^2_{\mathbb{H}^{\sigma}(\Omega)} + T^{2 \mu \gamma} \| f\|^2_{L^2(0,T;\mathbb{H}^{\sigma - s + 2 \mu s}(\Omega))}\bigg),
\end{multline*}
where we have used that $\log (1+z)\lesssim z^\mu$ for all $z\ge 0$ and $\mu>0$. This yields \eqref{eq:reg_in_y_gamma_1}. The estimates \eqref{eq:reg_in_y_gamma_2} and \eqref{eq:reg_in_y_gamma_3} follow, on the basis of \eqref{eq:nablaxpartial_y_l+1} and \eqref{eq:deltaxpartial_y_l+1}, respectively, by using similar arguments. 
\end{proof}

We conclude by studying the space regularity when $\gamma = 2$.

\begin{corollary}[space regularity for $\gamma = 2$]
\label{TH:regularitygamma2}
Let $\ue$ solve \eqref{eq:wave_alpha_extension}--\eqref{eq:initial_cond} for $s\in(0,1)$ and $\gamma = 2$. Let $0 \leq \sigma < s$ and $0 \leq \nu < 1+s $. Then, there exists $\kappa > 1$ such that the following regularity estimates hold: 
\begin{multline}
\| \partial_{y}^{\ell +1 } \ue \|_{L^2(0,T;L^2(\omega_{\alpha + 2 \ell -2\sigma,\theta},\C))}^2  
\lesssim  \kappa^{2(\ell+1)}(\ell+1)!^2 \big(
\| g \|^2_{\mathbb{H}^{\sigma + s}(\Omega)} + 
\| h \|^2_{\mathbb{H}^{\sigma}(\Omega)} 
\\
+ 
\| f \|_{L^2(0,T;\mathbb{H}^{\sigma}(\Omega))}^2 \big),
\label{eq:reg_in_y_gamma_1_gammaeq2} 
\end{multline} 
\begin{multline}
\| \nabla_{x'} \partial_{y}^{\ell +1 } \ue \|_{L^2(0,T;L^2(\omega_{\alpha + 2 (\ell+1) -2\nu,\theta},\C))}^2  
\lesssim  \kappa^{2(\ell+1)}(\ell+1)!^2 \big(
\| g \|^2_{\mathbb{H}^{\nu + s}(\Omega)} + 
\| h \|^2_{\mathbb{H}^{\nu}(\Omega)} 
\\
+ 
\| f \|_{L^2(0,T;\mathbb{H}^{\nu }(\Omega))}^2 \big).
\label{eq:reg_in_y_gamma_2_gammaeq2} 
\end{multline} 
and
\begin{multline}
\| \mathcal{L}_{x'} \partial_{y}^{\ell +1 } \ue \|_{L^2(0,T;L^2(\omega_{\alpha + 2 (\ell+1) -2\nu,\theta},\C))}^2  
\lesssim  \kappa^{2(\ell+1)}(\ell+1)!^2 \big(
\| g \|^2_{\mathbb{H}^{1+\nu + s}(\Omega)} \\ + 
\| h \|^2_{\mathbb{H}^{1+\nu}(\Omega)} 
+ 
\| f \|_{L^2(0,T;\mathbb{H}^{1+\nu }(\Omega))}^2 \big).
\label{eq:reg_in_y_gamma_3_gammaeq2} 
\end{multline} 
\end{corollary}
\begin{proof}
Once again, in light of Theorem~\ref{thm:pointwisebounds}, it suffices to bound $\| u_k \|_{L^2(0,T)}$. In this case ($\gamma = 2$) using \eqref{eq:var_parameters} we have that
\[
  \| u_k \|_{L^2(0,T)}^2 \lesssim T g_k^2 + T \lambda_k^{-s} h_k^2 + T^2 \lambda_k^{-s} \| f_k \|_{L^2(0,T)}^2.
\]
This estimate, together with \eqref{eq:partial_y_l+1}--\eqref{eq:deltaxpartial_y_l+1} yield the claimed bounds.
\end{proof}

\subsection{Space--time regularity}
\label{sub:spacetime}

The techniques and ideas used to obtain temporal and spatial regularity can be combined in order to obtain mixed regularity results, \ie we measure time derivatives in Sobolev norms of higher order. For brevity, we only present the results for $\gamma \in (1,2)$, but these can be extended to $\gamma = 2$ as well. The following regularity estimates are obtained by a direct combination of Theorem~\ref{thm:time_regularity} and Corollary~\ref{TH:regularity}.

\begin{corollary}[space--time regularity]
\label{cor:spacetimeregularity}
Let $\ue$ solve \eqref{eq:wave_alpha_extension}--\eqref{eq:initial_cond} for $s\in(0,1)$ and $\gamma \in (1,2)$. Let $0 \leq \sigma < s$, $0 \leq \nu < 1+s $ and $\rho >5-2\gamma$. Then, there exists $\kappa > 1$ such that the following regularity estimates hold: 
\begin{multline*}
\|t^{\rho/2} \partial_t^3 \partial_{y}^{\ell +1 } \ue \|_{L^2(0,T;L^2(\omega_{\alpha + 2 \ell -2\sigma,\theta},\C))}^2  
\lesssim  \kappa^{2(\ell+1)}(\ell+1)!^2 
\big(
\| g \|^2_{\mathbb{H}^{\sigma + 3s}(\Omega)} + 
\| h \|^2_{\mathbb{H}^{\sigma+2s}(\Omega)} 
\\
+ 
\| f \|_{H^2(0,T;\mathbb{H}^{\sigma+s}(\Omega))}^2 \big),
\end{multline*} 
\begin{multline*}
\|t^{\rho/2} \partial_t^3 \nabla_{x'} \partial_{y}^{\ell +1 } \ue \|_{L^2(0,T;L^2(\omega_{\alpha + 2 (\ell+1) -2\nu,\theta},\C))}^2  
\lesssim  \kappa^{2(\ell+1)}(\ell+1)!^2 
\big(
\| g \|^2_{\mathbb{H}^{\nu + 3s}(\Omega)} \\ + 
\| h \|^2_{\mathbb{H}^{\nu+2s}(\Omega)} 
+ 
\| f \|_{H^2(0,T;\mathbb{H}^{\nu + s}(\Omega))}^2 \big).
\end{multline*} 
and
\begin{multline*}
\|t^{\rho/2} \partial_t^3 \mathcal{L}_{x'} \partial_{y}^{\ell +1 } \ue \|_{L^2(0,T;L^2(\omega_{\alpha + 2 (\ell+1) -2\nu,\theta},\C))}^2  
\lesssim  \kappa^{2(\ell+1)}(\ell+1)!^2 
\big(
\| g \|^2_{\mathbb{H}^{1+\nu + 3s}(\Omega)} \\ + 
\| h \|^2_{\mathbb{H}^{1+\nu+2s}(\Omega)} 
+ 
\| f \|_{H^2(0,T;\mathbb{H}^{1+\nu+s}(\Omega))}^2 \big).
\end{multline*} 
\end{corollary}
\begin{proof}
We proceed first as in the proof of Theorem~\ref{thm:pointwisebounds} to obtain that
\[
  \|t^{\rho/2} \partial_t^3 \partial_{y}^{\ell +1 } \ue (\cdot,t) \|_{L^2(\omega_{\alpha + 2 \ell -2\sigma,\theta},\C))}^2 \lesssim (\ell+1)!^2 \kappa^{2(\ell+1)}
  t^\rho \sum_{k=1}^\infty \lambda_k^{\sigma +s} \left( \frac{\diff^3}{\diff t^3} u_k(t) \right)^2.
\]
In the case $f\equiv0$ the estimates of Lemma~\ref{lemma:G_gamma} with $r= \sigma +s$ then yield that
\[
  \int_0^T t^\rho \left\| \frac{\diff^3}{\diff t^3} u_k(t) \right\|_{\mathbb{H}^{\sigma+s}(\Omega)}^2 \diff t \lesssim \| g \|_{\mathbb{H}^{\sigma+3s}(\Omega)}^2 + \| h \|_{\mathbb{H}^{\sigma+2s}(\Omega)}^2.
\]
The case $g \equiv 0$, $h \equiv 0$ can be obtained as in the proof of Theorem~\ref{thm:time_regularity}, Step 3, to obtain
\[
  \int_0^T t^\rho \left\| \frac{\diff^3}{\diff t^3} u_k(t) \right\|_{\mathbb{H}^{\sigma+s}(\Omega)}^2 \diff t \lesssim \| f \|_{H^2(0,T;\mathbb{H}^{\sigma+s}(\Omega))}^2.
\]

The other estimates follow the same line of reasoning.
\end{proof}

\subsection{Application: error estimates for fully discrete schemes}
\label{sub:application}

Our main motivation to study the regularity of the solution to \eqref{eq:fractional_wave} and \eqref{eq:wave_alpha_extension}--\eqref{eq:initial_cond} is to provide error estimates for numerical methods. Here we sketch how the estimates we obtained in previous sections fit into this program. 

We begin by introducing a family of finite dimensional spaces $\V_h \subset \HL(y^\alpha, \C)$, parametrized by $h>0$. We also introduce
\[
 \mathcal{T} =\{t_j\}_{j=0}^J, \quad 0 = t_0 < t_1 < \ldots < t_J = T,
\]
a partition of the time interval $[0,T]$. We denote $\tau = \max\{t_{j+1} - t_j: j = 0, \ldots, J-1 \}$. 

A fully discrete scheme then seeks for $\ue_h^\tau = \{ \ue_h^j \}_{j=0}^J \subset \V_h$ such that, for every $j \geq 2$,
\begin{equation}
\label{eq:weak_wave_full_discr}
  \langle \tr \delta_\tau^\gamma \ue_h^j, \tr V \rangle + a( \ue_h^j, V) = \langle f, \tr V \rangle, \quad \forall V \in \V_h.
\end{equation}
In addition, we require the $\ue_h^0$ and $\ue_h^1$ are determined by (suitable approximations of) the initial data $g$ and $h$.

In \eqref{eq:weak_wave_full_discr} we introduced the mapping $\delta_\tau^\gamma$ that acts on (time) sequences. We assume that $\delta_\tau^\gamma$ is \emph{consistent} with $\partial_t^\gamma$ in the sense that there is $\Wcal \subset L^2(0,T;\Hsd)$ such that, if $w \in \Wcal$ then
\begin{equation}
\label{eq:consistent}
  \| \delta_\tau^\gamma w - \partial_t^\gamma w \|_{L^2(0,T;\Hsd)} \lesssim \mathcal{E}_t(\tau) \| w \|_\Wcal,
\end{equation}
for some function $\mathcal{E}_t: \R_+ \to \R_+$ such that $\mathcal{E}_t(\tau) \downarrow 0$ as $\tau \downarrow 0$. Notice that Theorem~\ref{thm:time_regularity} provides particular instances of $\Wcal$. We, in addition, need to assume the following bound
\begin{equation}
\label{eq:weirdbound}
  \| \delta_\tau^\gamma w \|_{L^2(0,T;\Hsd)} \lesssim \| \partial_t^k w \|_{L^2(0,T;\Hsd)},
\end{equation}
for $k=2$ or $3$.

We also assume that the scheme \eqref{eq:weak_wave_full_discr} is \emph{stable}, \ie that its solution satisfies
\begin{equation}
\label{eq:stable}
  \| \ue_h^\tau \|_{L^2(0,T;\HLn(y^\alpha,\C))} \lesssim \| f \|_{L^2(0,T;\Hsd)}.
\end{equation}

The final ingredient necessary for the analysis of \eqref{eq:weak_wave_full_discr} is the so-called \emph{Galerkin projection} $\calG_h$: 
\[
 \calG_h: \HL(y^\alpha,\C) \to \V_h, \qquad a(\calG_h w, V) = a(w,V), \quad \forall V \in \V_h.
\]
This immediately gives its stability and the fact that it has quasi-best approximation properties. From this, by a proper choice of $\V_h$, it follows that there is $\Zcal \subset \HL(y^\alpha,\C)$ such that if $w \in \Zcal$ we have
\begin{equation}
\label{eq:Galprojestimate}
 \| w - \calG_h w \|_{\HLn(y^\alpha,\C)} \lesssim \mathcal{E}_x(h) \| w \|_\Zcal,
\end{equation}
for a function $\mathcal{E}_x : \R_+ \to \R_+$ such that $\mathcal{E}_x(h) \downarrow 0$ as $h \downarrow 0$. Notice that \eqref{eq:Galprojestimate} can be integrated in time to obtain error estimates in $L^2(0,T;\HL(y^\alpha,\C))$. In this case, Corollary~\ref{TH:regularity} provides particular instances of $L^2(0,T;\Zcal)$.

\begin{remark}[choice of $\V_h$ and $\delta_\tau^\gamma$]
\label{rem:numerics}
The choice of discrete space $\V_h$ and operator $\delta_\tau^\gamma$ satisfying the requisite properties is by no means trivial and is at the heart of the numerical analysis of \eqref{eq:weak_wave}. Our sole purpose here is to show how the regularity estimates that we have provided can be used.
\end{remark}

With all these tools at hand, we can obtain an error analysis for \eqref{eq:weak_wave_full_discr} as follows. Define the error 
\[
 e = \ue - \ue_h^\tau = (\ue - \calG_h \ue) + (\calG_h \ue - \ue_h^\tau) = E_\ue + e_h,
\]
where $E_\ue$ is the so-called interpolation error and $e_h$ the approximation error. From \eqref{eq:Galprojestimate} it immediately follows that
\[
  \| E_\ue \|_{L^2(0,T;\HLn(y^\alpha,\C))} \lesssim \mathcal{E}_x(h) \| \ue \|_{L^2(0,T;\Zcal)},
\]
so that it suffices to bound $e_h$. Set, in \eqref{eq:weak_wave} $v = V \in \V_h$ and subtract from it \eqref{eq:weak_wave_full_discr}. This yields
\[
  \langle \tr \delta_\tau^\gamma e_h^j, \tr V \rangle + a( e_h^j, V) = \langle \tr (\delta_\tau^\gamma \calG_h \ue - \partial_t^\gamma \ue), \tr V \rangle \quad \forall V \in \V_h.
\]
The stability of the scheme, expressed in \eqref{eq:stable}, then yields
\begin{multline*}
  \| e_h \|_{L^2(0,T;\HLn(y^\alpha,\C))} \lesssim \| \tr (\delta_\tau^\gamma \calG_h \ue - \partial_t^\gamma \ue) \|_{L^2(0,T;\Hsd)} \\ \leq
    \| \tr (\delta_\tau^\gamma - \partial_t^\gamma) \ue \|_{L^2(0,T;\Hsd)} + \| \delta_\tau^\gamma (I-\calG_h ) \ue\|_{L^2(0,T;\Hsd)}.
\end{multline*}
The consistency, expressed in \eqref{eq:consistent} allows us to bound
\[
  \| \tr (\delta_\tau^\gamma - \partial_t^\gamma) \ue \|_{L^2(0,T;\Hsd)} \lesssim \mathcal{E}_t(\tau) \| \ue \|_\Wcal.
\]
Finally, the bound \eqref{eq:weirdbound} and the approximation properties of $\calG_h$, given in \eqref{eq:Galprojestimate}, yield
\[
  \| \delta_\tau^\gamma (I-\calG_h ) \ue\|_{L^2(0,T;\Hsd)} \lesssim \| (I-\calG_h ) \partial_t^k \ue \|_{L^2(0,T;\Hsd)} \lesssim \mathcal{E}_x(h) \| \partial_t^k \ue \|_{L^2(0,T;\Zcal)}.
\]
Notice that Corollary~\ref{cor:spacetimeregularity} gives conditions so that $\| \partial_t^k \ue \|_{L^2(0,T;\Zcal)} < \infty$.

Conclude by gathering all the estimates given above.

\bibliographystyle{plain}
\bibliography{biblio}

\def\cprime{$'$} \def\cprime{$'$} \def\cprime{$'$} \def\cprime{$'$}
  \def\cprime{$'$} \def\cprime{$'$}
\begin{thebibliography}{10}

\bibitem{Abra}
M.~Abramowitz and I.A. Stegun.
\newblock {\em Handbook of mathematical functions with formulas, graphs, and
  mathematical tables}.
\newblock National Bureau of Standards Applied Mathematics Series. 1964.

\bibitem{MR1217081}
N.I. Achieser.
\newblock {\em Theory of approximation}.
\newblock Dover Publications, Inc., New York, 1992.
\newblock Translated from the Russian and with a preface by Charles J. Hyman,
  Reprint of the 1956 English translation.

\bibitem{Adams}
R.A. Adams.
\newblock {\em Sobolev spaces}.
\newblock Academic Press [A subsidiary of Harcourt Brace Jovanovich,
  Publishers], New York-London, 1975.
\newblock Pure and Applied Mathematics, Vol. 65.

\bibitem{MR2600998}
I.~Athanasopoulos and L.A. Caffarelli.
\newblock Continuity of the temperature in boundary heat control problems.
\newblock {\em Adv. Math.}, 224(1):293--315, 2010.

\bibitem{BMNOSS:17}
L.~Banjai, J.M. Melenk, R.H. Nochetto, E.~Ot\'arola, A.J. Salgado, and Ch.
  Schwab.
\newblock Tensor {FEM} for spectral fractional diffusion.
\newblock {\em arXiv:1707.07367v1}, 2017.

\bibitem{BS}
M.{\v{S}}. Birman and M.Z. Solomjak.
\newblock {\em Spektralnaya teoriya samosopryazhennykh operatorov v gilbertovom
  prostranstve}.
\newblock Leningrad. Univ., Leningrad, 1980.

\bibitem{MR3393253}
M.~Bonforte, Y.~Sire, and J.L. V\'azquez.
\newblock Existence, uniqueness and asymptotic behaviour for fractional porous
  medium equations on bounded domains.
\newblock {\em Discrete Contin. Dyn. Syst.}, 35(12):5725--5767, 2015.

\bibitem{CT:10}
X.~Cabr\'e and J.~Tan.
\newblock Positive solutions of nonlinear problems involving the square root of
  the {L}aplacian.
\newblock {\em Adv. Math.}, 224(5):2052--2093, 2010.

\bibitem{CS:07}
L.~Caffarelli and L.~Silvestre.
\newblock An extension problem related to the fractional {L}aplacian.
\newblock {\em Comm. Partial Differential Equations}, 32(7-9):1245--1260, 2007.

\bibitem{CDDS:11}
A.~Capella, J.~D\'avila, L.~Dupaigne, and Y.~Sire.
\newblock Regularity of radial extremal solutions for some non-local semilinear
  equations.
\newblock {\em Comm. Partial Differential Equations}, 36(8):1353--1384, 2011.

\bibitem{MR2737788}
A.~de~Pablo, F.~Quir\'os, A.~Rodr\'iguez, and J.L. V\'azquez.
\newblock A fractional porous medium equation.
\newblock {\em Adv. Math.}, 226(2):1378--1409, 2011.

\bibitem{MR2954615}
A.~de~Pablo, F.~Quir\'os, A.~Rodr\'iguez, and J.L. V\'azquez.
\newblock A general fractional porous medium equation.
\newblock {\em Comm. Pure Appl. Math.}, 65(9):1242--1284, 2012.

\bibitem{MR3549087}
M.~Dehghan, M.~Abbaszadeh, and A.~Mohebbi.
\newblock Analysis of a meshless method for the time fractional diffusion-wave
  equation.
\newblock {\em Numer. Algorithms}, 73(2):445--476, 2016.

\bibitem{MR2651618}
R.~Du, W.R. Cao, and Z.Z. Sun.
\newblock A compact difference scheme for the fractional diffusion-wave
  equation.
\newblock {\em Appl. Math. Model.}, 34(10):2998--3007, 2010.

\bibitem{MR1800316}
J.~Duoandikoetxea.
\newblock {\em Fourier analysis}, volume~29 of {\em Graduate Studies in
  Mathematics}.
\newblock American Mathematical Society, Providence, RI, 2001.
\newblock Translated and revised from the 1995 Spanish original by David
  Cruz-Uribe.

\bibitem{Evans}
L.C. Evans.
\newblock {\em Partial differential equations}, volume~19 of {\em Graduate
  Studies in Mathematics}.
\newblock American Mathematical Society, Providence, RI, second edition, 2010.

\bibitem{ShinChan}
D.~Fujiwara.
\newblock Concrete characterization of the domains of fractional powers of some
  elliptic differential operators of the second order.
\newblock {\em Proc. Japan Acad.}, 43:82--86, 1967.

\bibitem{GU}
V.~Gol$\prime$dshtein and A.~Ukhlov.
\newblock Weighted {S}obolev spaces and embedding theorems.
\newblock {\em Trans. Amer. Math. Soc.}, 361(7):3829--3850, 2009.

\bibitem{MR3244285}
R.~Gorenflo, A.A. Kilbas, F.~Mainardi, and S.V. Rogosin.
\newblock {\em Mittag-{L}effler functions, related topics and applications}.
\newblock Springer Monographs in Mathematics. Springer, Heidelberg, 2014.

\bibitem{Grisvard}
P.~Grisvard.
\newblock {\em Elliptic problems in nonsmooth domains}, volume~69 of {\em
  Classics in Applied Mathematics}.
\newblock Society for Industrial and Applied Mathematics (SIAM), Philadelphia,
  PA, 2011.
\newblock Reprint of the 1985 original [ MR0775683], With a foreword by Susanne
  C. Brenner.

\bibitem{HKM}
J.~Heinonen, T.~Kilpel{\"a}inen, and O.~Martio.
\newblock {\em Nonlinear potential theory of degenerate elliptic equations}.
\newblock Oxford Mathematical Monographs. The Clarendon Press Oxford University
  Press, New York, 1993.
\newblock Oxford Science Publications.

\bibitem{MR3471180}
V.R. Hosseini, E.~Shivanian, and W.~Chen.
\newblock Local radial point interpolation ({MLRPI}) method for solving time
  fractional diffusion-wave equation with damping.
\newblock {\em J. Comput. Phys.}, 312:307--332, 2016.

\bibitem{MR2870026}
X.~Hu and L.~Zhang.
\newblock On finite difference methods for fourth-order fractional
  diffusion-wave and subdiffusion systems.
\newblock {\em Appl. Math. Comput.}, 218(9):5019--5034, 2012.

\bibitem{MR0145254}
J.-P. Kahane.
\newblock {\em Teor\'ia constructiva de funciones}.
\newblock Universidad de Buenos Aires, Buenos Aires, 1961.

\bibitem{MR3613323}
Y.~Kian and M.~Yamamoto.
\newblock On existence and uniqueness of solutions for semilinear fractional
  wave equations.
\newblock {\em Fract. Calc. Appl. Anal.}, 20(1):117--138, 2017.

\bibitem{fractional_book}
A.A. Kilbas, H.M. Srivastava, and J.J. Trujillo.
\newblock {\em Theory and applications of fractional differential equations}.
\newblock Elsevier Science B.V., Amsterdam, 2006.

\bibitem{Kufner80}
A.~Kufner.
\newblock {\em Weighted {S}obolev spaces}, volume~31 of {\em Teubner-Texte zur
  Mathematik [Teubner Texts in Mathematics]}.
\newblock BSB B. G. Teubner Verlagsgesellschaft, Leipzig, 1980.
\newblock With German, French and Russian summaries.

\bibitem{KO84}
A.~Kufner and B.~Opic.
\newblock How to define reasonably weighted {S}obolev spaces.
\newblock {\em Comment. Math. Univ. Carolin.}, 25(3):537--554, 1984.

\bibitem{Landkof}
N.S. Landkof.
\newblock {\em Foundations of modern potential theory}.
\newblock Springer-Verlag, New York-Heidelberg, 1972.
\newblock Translated from the Russian by A. P. Doohovskoy, Die Grundlehren der
  mathematischen Wissenschaften, Band 180.

\bibitem{MR3109799}
L.~Li, D.~Xu, and M.~Luo.
\newblock Alternating direction implicit {G}alerkin finite element method for
  the two-dimensional fractional diffusion-wave equation.
\newblock {\em J. Comput. Phys.}, 255:471--485, 2013.

\bibitem{MR3465296}
Z.~Li, O.Y. Imanuvilov, and M.~Yamamoto.
\newblock Uniqueness in inverse boundary value problems for fractional
  diffusion equations.
\newblock {\em Inverse Problems}, 32(1):015004, 16, 2016.

\bibitem{Lions}
J.-L. Lions and E.~Magenes.
\newblock {\em Non-homogeneous boundary value problems and applications. {V}ol.
  {I}}.
\newblock Springer-Verlag, New York-Heidelberg, 1972.
\newblock Translated from the French by P. Kenneth, Die Grundlehren der
  mathematischen Wissenschaften, Band 181.

\bibitem{MR2595950}
Y.~Luchko.
\newblock Some uniqueness and existence results for the initial-boundary-value
  problems for the generalized time-fractional diffusion equation.
\newblock {\em Comput. Math. Appl.}, 59(5):1766--1772, 2010.

\bibitem{McLean}
W.~McLean.
\newblock {\em Strongly elliptic systems and boundary integral equations}.
\newblock Cambridge University Press, Cambridge, 2000.

\bibitem{M:10}
W.~McLean.
\newblock Regularity of solutions to a time-fractional diffusion equation.
\newblock {\em ANZIAM J.}, 52(2):123--138, 2010.

\bibitem{MR1225703}
W.~McLean and V.~Thom\'ee.
\newblock Numerical solution of an evolution equation with a positive-type
  memory term.
\newblock {\em J. Austral. Math. Soc. Ser. B}, 35(1):23--70, 1993.

\bibitem{MFSV:17}
D.~Meidner, J.~Pfefferer, K.~Sch\"urholz, and B.~Vexler.
\newblock $hp$-finite elements for fractional diffusion.
\newblock {\em arXiv:1706.04066v1}, 2017.

\bibitem{Muckenhoupt}
B.~Muckenhoupt.
\newblock Weighted norm inequalities for the {H}ardy maximal function.
\newblock {\em Trans. Amer. Math. Soc.}, 165:207--226, 1972.

\bibitem{NOS}
R.H. Nochetto, E.~Ot{\'a}rola, and A.J. Salgado.
\newblock A {PDE} approach to fractional diffusion in general domains: a priori
  error analysis.
\newblock {\em Found. Comput. Math.}, 15(3):733--791, 2015.

\bibitem{MR0435697}
F.W.J. Olver.
\newblock {\em Asymptotics and special functions}.
\newblock Academic Press [A subsidiary of Harcourt Brace Jovanovich,
  Publishers], New York-London, 1974.
\newblock Computer Science and Applied Mathematics.

\bibitem{Otarola}
E.~Ot\'arola.
\newblock {\em {A PDE approach to numerical fractional diffusion}}.
\newblock PhD thesis, University of Maryland, College Park, 2014.

\bibitem{Podlubny}
I.~Podlubny.
\newblock {\em Fractional differential equations}, volume 198 of {\em
  Mathematics in Science and Engineering}.
\newblock Academic Press, Inc., San Diego, CA, 1999.

\bibitem{MR3014456}
T.~Roub{\'i}{\v c}ek.
\newblock {\em Nonlinear partial differential equations with applications},
  volume 153 of {\em International Series of Numerical Mathematics}.
\newblock Birkh\"auser/Springer Basel AG, Basel, second edition, 2013.

\bibitem{Sakayama}
K.~Sakamoto and M.~Yamamoto.
\newblock Initial value/boundary value problems for fractional diffusion-wave
  equations and applications to some inverse problems.
\newblock {\em J. Math. Anal. Appl.}, 382(1):426--447, 2011.

\bibitem{Samko}
S.G. Samko, A.A. Kilbas, and O.I. Marichev.
\newblock {\em Fractional integrals and derivatives}.
\newblock Gordon and Breach Science Publishers, Yverdon, 1993.

\bibitem{MR2270163}
L.~Silvestre.
\newblock Regularity of the obstacle problem for a fractional power of the
  {L}aplace operator.
\newblock {\em Comm. Pure Appl. Math.}, 60(1):67--112, 2007.

\bibitem{ST:10}
P.R. Stinga and J.L. Torrea.
\newblock Extension problem and {H}arnack's inequality for some fractional
  operators.
\newblock {\em Comm. Part. Diff. Eqs.}, 35(11):2092--2122, 2010.

\bibitem{MR2200938}
Z.-Z. Sun and X.~Wu.
\newblock A fully discrete difference scheme for a diffusion-wave system.
\newblock {\em Appl. Numer. Math.}, 56(2):193--209, 2006.

\bibitem{Tartar}
L.~Tartar.
\newblock {\em An introduction to {S}obolev spaces and interpolation spaces},
  volume~3 of {\em Lecture Notes of the Unione Matematica Italiana}.
\newblock Springer, Berlin, 2007.

\bibitem{Turesson}
B.O. Turesson.
\newblock {\em Nonlinear potential theory and weighted {S}obolev spaces},
  volume 1736 of {\em Lecture Notes in Mathematics}.
\newblock Springer-Verlag, Berlin, 2000.

\bibitem{MR3192423}
J.L. V\'azquez and B.~Volzone.
\newblock Symmetrization for linear and nonlinear fractional parabolic
  equations of porous medium type.
\newblock {\em J. Math. Pures Appl. (9)}, 101(5):553--582, 2014.

\bibitem{MR2970754}
Y.-N. Zhang, Z.-Z. Sun, and X.~Zhao.
\newblock Compact alternating direction implicit scheme for the two-dimensional
  fractional diffusion-wave equation.
\newblock {\em SIAM J. Numer. Anal.}, 50(3):1535--1555, 2012.

\bibitem{MR3304331}
X.~Zhao and Z.-Z. Sun.
\newblock Compact {C}rank-{N}icolson schemes for a class of fractional
  {C}attaneo equation in inhomogeneous medium.
\newblock {\em J. Sci. Comput.}, 62(3):747--771, 2015.

\end{thebibliography}

\end{document}